%% file: main.tex
\documentclass[11pt]{article}
\usepackage{amsmath}
\usepackage{ulem}
\usepackage[toc]{appendix}

  \usepackage{graphics}
  \usepackage{amsfonts}
  \usepackage{epsfig}
  \usepackage{float}
  \usepackage{graphicx}
 \usepackage{epstopdf}
 \usepackage{amssymb}
 \usepackage{marginnote}
 \usepackage{mathrsfs}
 \usepackage{amsthm}
 \usepackage[colorlinks=true]{hyperref}
 \usepackage[numbers]{natbib}
\usepackage{color}
\usepackage{dsfont}
\hypersetup{urlcolor=blue, citecolor=red}
\usepackage{hyperref}
\hypersetup{%
  colorlinks = true,
  linkcolor  = black
}

\numberwithin{equation}{section}

\newtheorem{theorem}{Theorem}[section]
\newtheorem{corollary}{Corollary}[section]
\newtheorem{lemma}[theorem]{Lemma}

\newtheorem{proposition}{Proposition}[section]

\newtheorem{definition}[theorem]{Definition}

\newtheorem{remark}[theorem]{Remark}

\def\bar{\overline}

\DeclareMathOperator\MF{{\mathcal{MF}}}
\DeclareMathOperator\PMF{{\mathcal{PMF}}}
\DeclareMathOperator\M{Mod}
\DeclareMathOperator\T{Teich}
\DeclareMathOperator\E{Ext}
\DeclareMathOperator\SE{Sect}

\DeclareMathOperator\PGM{\partial_{GM}\overline{\T(S)}^{GM}}
\DeclareMathOperator\GM{\overline{\T(S)}^{GM}}

\begin{document}

\title{Dynamics of subgroups of mapping class groups}
\author{Ilya Gekhtman \and Biao Ma}

\date{}
\maketitle
\begin{abstract}
Let $\M(S)$ be the mapping class group of a closed orientable surface $S$ of genus $g \geq 2$. Let $G$ be a non-elementary subgroup of $\M(S)$ so that the associated Bowen-Margulis measure is finite. In this paper, we give an asymptotic growth formula for $G$ with respect to the
Teichm\"{u}ller metric.  
\end{abstract}
\section{Introduction}
Given an countable group $G$ acting properly discontinuously by isometries on a proper metric space $X$, a natural quantity to study is the asymptotic orbit growth of the action, i.e. the number $N(R,x)$ of elements displacing a basepoint $x$ by at most a given number $R$. When $X$ is a negatively curved contractible manifold and $M=X/\Gamma$ is compact, Margulis \cite{Margulis} proved in his 1969 thesis  that $N(R,x)/e^{\delta R}\to_{R\to \infty} C_x$ where $C_x>0$ is independent of $R$ and $\delta$ is the topological entropy of the geodesic flow on $M$. Roblin \cite{Roblin} proved the same asymptotics, with the cocompactness assumption replaced by the assumption that the so called \textit{Bowen-Margulis} measure on the unit tangent bundle of $M$ is finite, a condition which covers for instance finite volume and geometrically finite rank 1 locally symmetric spaces.
The study of the dynamics of the action of the mapping class group on Teichm$\ddot{\mathrm{u}}$ller
space has long been influenced by analogy with the actions of discrete isometry groups of manifolds of negative curvature. Although the Teichm$\ddot{\mathrm{u}}$ller space of higher genus surfaces is not negatively curved in any meaningful global sense (e.g. it is not $CAT(0)$ and not Gromov hyperbolic), in many ways it behaves ``asymptotically" or ``on average" like a negatively curved space and the mapping class group as a lattice in such a space. It is interesting to exploit these analogies to  extend the asymptotic orbit growth results of Margulis and Roblin to the setting of the mapping class group and its subgroups acting on Teichm$\ddot{\mathrm{u}}$ller space; this is the goal of this paper. 

We now state the main result of this paper. Let $S = S_g$ be a connected closed surface of genus $g \geq 2$. Our results also hold for surfaces with punctures but we only state it for closed surfaces. The mapping class group $\M(S)$ of $S$ is the group of isotopy classes of orientation-preserving homeomorphisms of $S$. A subgroup $G<\M(S)$ of the mapping class group is called nonelementary if it contains two independent pseuo-Anosov mapping classes. For a nonelementary subgroup $G$ of $\M(S)$ such that the Teichmueller geodesic flow on the unit cotangent bundle $QD^1(S)/G$ of $\T(S)/G$ has finite Bowen-Margulis measure (see Section 3 for the precise definition), we prove 
\begin{theorem}[Theorem \ref{main:counting}]\label{theorem:introcounting}
   Let $S$ be a closed surface of genus $g$ and $G$ be a non-elementary subgroup of the mapping class group $\M(S)$ so that the Bowen-Margulis measure $\mu_G$ on $QD^1(S)/G$ is finite. Let $x,y$ be two points in the Teichm\"{u}ller space $\T(S)$ and denote the critical exponent of $G$ by $\delta_G$.  We then have
    $$\lim_{R \to \infty}\frac{\sharp\{g \in G: d(x, gy) \leq R\}}{e^{\delta_G R}} =c,$$ where $c = \frac{1}{\mu_G(QD^1(S)/G)}$ and $d$ is the Teichm\"{u}ller metric.     
\end{theorem}

In two special cases, the results of this paper were 
established earlier. One is when $G$ is the full mapping class group $\M(S)$; in this case the results were established by Athreya-Bufetov-Eskin-Mirzakhani \cite{ABEM}. The other, when $G$ is a convex-cocompact subgroup with at least one axis in the principal stratum, was studied by the first named author in his thesis and unpublished preprint \cite{gekhtman2013dynamics}. 
The result of \cite{ABEM} uses in an essential way the $SL_2 \mathbb{R}$ homogeneity of the measure of maximal entropy (Masur-Veech measure) of the geodesic flow whereas the result in \cite{gekhtman2013dynamics} uses in an essential way the Gromov hyperbolicity of the acting group $G$, neither of which occurs our setting.
Furthermore,i0 the results of \cite{ABEM} and \cite{gekhtman2013dynamics} both use in an essential way some $\rm{CAT(-1)}$-like properties of the Teichm$\ddot{\mathrm{u}}$ller metric, namely that two geodesics in Teichm$\ddot{\mathrm{u}}$ller space which fellow travel for a long time over compact parts of moduli space must get very close in the middle. This requires the use of some complicated analytic objects, such as the modified Hodge norm. This is in fact the reason for the assumption in \cite{gekhtman2013dynamics} that some pseudo-Anosov axis must lie in the principal stratum. In the present paper we unify, extend and simplify the proofs of the results of \cite{ABEM} and \cite{gekhtman2013dynamics}. In addition, our results will apply to a much larger class of groups including \textit{statistically convex cocompact} subgroups of mapping class groups considered by Yang \cite{Yang_scc} such as free products of the form $H \star <p>$ where $H$ is free abelian generated by Dehn twists and $p$ is a pseudo-Anosov element. We hope that they will also be applicable to mapping class groups of non-orientable surfaces studied by Khan \cite{Khan}, which can be realized as subgroups of mapping class groups of orientable surfaces preserving totally geodesic submanifolds of Teichm\"{u}ller space.
Inspired by and following the techniques of Link \cite{Link} in the rank 1 $\rm{CAT(0)}$ setting we do not use any $\rm{CAT(-1)}$-like of homogeneity properties of the Teichm$\ddot{\mathrm{u}}$ller metric, and therefore we expect our techniques to be applicable in the more general setting of actions with contracting elements (see e.g. \cite{Yang_scc} for details). 
To avoid using $\rm{CAT(-1)}$ properties of the Teichm$\ddot{\mathrm{u}}$ller metric, as well as to deal with global non-hyperbolicity of Teichm$\ddot{\mathrm{u}}$ller space, we had to make use of the horofunction compactification of $\T(S)$, the so called Gardiner-Masur boundary $\PGM$, instead of Thurston's space of projective measured foliations, which is better understood geometrically but less natural from a dynamical point of view. Some limit points of $G$ in $\PGM$ are non-Busemann, i.e. cannot be approximated by geodesic rays. To deal with these points we make use of so called optimal geodesics, introduced by Azemar in \cite{azemar2021qualitative} which can be thought of as approximating these points most efficiently. Various boundary "shadows" illuminated by subsets of Teichm$\ddot{\mathrm{u}}$ller space we will work shall be defined in terms of these optimal geodesics.

The preprint in \cite{gekhtman2013dynamics} will not be published. Therefore, to make this paper self-contained, we have included the proof of the most important result therein, which will guarantee in our setting the  mixing of the geodesic flow on $QD^1(S)/G$ with res
pect to the Bowen-Margulis measure. This result is a certain nondegeneracy
condition for the translation length spectrum of pseudo-Anosov elements of $G$. Recall that the spectrum of $G$ is the set of logarithms of the dilatations of pseudo-Anosov elements in $G$, ie the set of their translation lengths in Teichmueller space.
\begin{theorem}[Theorem \ref{thm:nonarithmetic}]\label{theorem:introarithem}
    Let $G$ be a non-elementary subgroup of $\M(S)$. Then the spectrum $Spec(G)$ generates a dense subgroup of $\mathds{R}$.
\end{theorem}

For sub-semigroups of $SL_{n}(\mathds{R})$ acting irreducibly on $\mathds{R}^n$ and containing a proximal
element, an analogous result is proved by Guivarch and Urban in \cite{GuivarUrban}. In variable
negative curvature this question remains open. We prove Theorem \ref{theorem:introarithem} by using the
affine and symplectic structure of $\MF(S)$ given by train track coordinates to embed a subsemigroup of $G$ into $SL_{n}(\mathds{R})$ with the image satisfying the conditions of \cite{GuivarUrban}.

\noindent{\textbf{Acknowledgments.}}
We thank Kasra Rafi for useful conversations. The second named author was supported by grant No. GIF I-1485-304.6/2019 of Amos Nevo.
\section{Backgrounds on Teichm\"{u}ller Theory}
In this section, we introduce some background on Teichm\"{u}ller theory which will be used in the sequel. See for instance \cite{FarMar}  for more details on mapping class groups. 

Let $S =S_{g}$ be a genus $g \geq 2$, closed, connected, orientable surface. All arguments here work for hyperbolic surfaces with punctures as well. The {\it mapping class group} $\M(S)$ of $S$ is the group of isotopy classes of orientation-preserving homeomorphisms of $S$. Namely, if the group of orientation-preserving homeomorphisms of $S$ is denoted by ${\rm Homeo}^{+}(S)$ and the group of homeomorphisms of $S$ that are isotopic to the identity is denoted by ${\rm Homeo}_{0}(S)$, then $$ \M(S) = {\rm Homeo}^{+}(S)/ {\rm Homeo}_{0}(S).$$
 The {\it Teichm\"{u}ller space} $\T(S)$ of $S$ is the space of homotopy classes of hyperbolic structures $[(X,\phi)]$. The mapping class group $\M(S)$ acts on $\T(S)$ isometrically with respect to the Teichm\"{u}ller distance $d$. The cotangent bundle of $\T(S)$ can be canonically identified with the bundle $QD(S)$ of holomorphic quadratic differentials on $S$. Each element $q \in QD(S)$ defines an area on $S$. The unit cotangent bundle of $\T(S)$ is then identified with the subset $QD^1(S)$ of $QD(S)$ consisting of area one elements in $QD(S)$. 
 
 \noindent{\textbf{Teichm\"{u}ller compactification.}}~~Let $\MF(S)$ be the space of measured foliations on $S$ and $\PMF(S)$ be the space of projective measured foliations. The space $\MF(S)$ is a very powerful tool for understanding the Teichm\"{u}ller space. In fact, to each $q \in QD(S)$ on a Riemann surface $X$, one can associate two measured foliations, namely the vertical measured foliation $\mathcal{V}(q)$ and the horizontal measured foliation $\mathcal{H}(q)$. This assignment gives a homeomorphism from $QD(S)$ onto its image in $\MF(S) \times \MF(S)$ (cf.\cite{HubMas}). Fixing a base point, the compactification of $\T(S)$ given by the geodesic rays passing through the base point is called the {\it Teichm\"{u}ller compactification} of $\T(S)$. The above assignment then gives a homeomorphism between the boundary and $\PMF(S)$ 

 The Teichm\"{u}ller compactification is very useful for many purposes. However, we need more than just a compactification. The Teichm\"{u}ller compactification depends on the choice of a base point $X$ and therefore the action of $\M(S)$ on $\overline{\T(S)}^{T}$ is not continuous \cite{Kerc}. Nevertheless, this action is still nice in the following sense. Let $\mathcal{UE}(S) \subset \PMF(S)$ be the set of projective classes of uniquely ergodic measured foliations on $S$. Since there is no ambiguity, we denote $\mathcal{UE}(S)$ by $\mathcal{UE}$. Let $\T(S) \cup_{T} \mathcal{UE}$ be the image of $\T(S) \cup \mathcal{UE}$ in $\overline{\T(S)}^{T}$ equipped with the induced topology. We then have
\begin{theorem}[Masur \cite{Masur}]
 The action of the mapping class group $\M(S)$ on the Teichm\"{u}ller space $\T(S)$ extends continuously to $\T(S) \cup_{T} \mathcal{UE}$. 
\end{theorem} 

 \noindent{\textbf{Gardiner-Masur compactification.}}~~For studying counting problems for mapping class groups, another compactification of $\T(S)$, namely the Gardiner-Masur compactification, seems to be much more suitable. We give a short introduction and collect certain properties of it that will be used later. The reader is referred to, for instance, \cite{Gardiner-Masur}, \cite{Liu-Su_horo} for a nice treatment.

 Denote by $\mathcal{C}$ the set of (isotopy classes of) simple closed curves on $S$ and $P(\mathbb{R}_{+}^{\mathcal{C}})$ the projective space of $\mathbb{R}_{+}^{\mathcal{C}}$, where $\mathbb{R}_{+}^{\mathcal{C}}$ is the set of maps from $\mathcal{C}$ to $\mathbb{R}_{+}$. Denote the natural quotient by $$\mathfrak{p}:\mathbb{R}_{+}^{\mathcal{C}} \to P(\mathbb{R}_{+}^{\mathcal{C}}).$$ The map $$\iota: X \mapsto (\E_{X}(\alpha)^{\frac{1}{2}})_{\alpha \in \mathcal{C}},$$ where $\E_X(\alpha)$ is the extremal length of $\alpha \in \mathcal{C}$ on $X \in \T(S)$, defines an embedding of $\T(S)$ into $\mathbb{R}_{+}^{\mathcal{C}}$. Thus the composition $\mathfrak{p} \circ \iota$ defines a map from $\T(S)$ to $P(\mathbb{R}_{+}^{\mathcal{C}})$. Gardiner and Masur showed that the map $\mathfrak{p} \circ \iota$ is also an embedding and has a compact closure in $P(\mathbb{R}_{+}^{\mathcal{C}})$. We identify $\T(S)$ with $\mathfrak{p} \circ \iota (\T(S))$. The closure $\GM$ in $P(\mathbb{R}_{+}^{\mathcal{C}})$ is called the \textit{Gardiner-Masur compactification} and the boundary $\PGM = \GM - \T(S)$ is therefore called the \textit{Gardiner-Masur boundary} of $\T(S)$.


 The Gardiner-Masur compactification of $\T(S)$ has the following properties.
 \begin{enumerate}
 
     \item[(A.1)] The Gardiner-Masur compactification of $\T(S)$ is homeomorphic to the \textit{horofunction compactification} of $\T(S)$ with respect to the Teichm\"{u}ller metric \cite{Liu-Su_horo}, \cite{Walsh}. The horofunction compactification can be defined for any metric space and is widely used in metric geometry. 
     \item[(A.2)] The action of $\M(S)$ on $\T(S)$ extends continuously to an action on $\GM$ \cite{Liu-Su_horo}. Therefore this compactification is different from the Teichm\"{u}ller compactification $\overline{\T(S)}^{T}$.
     \item[(A.3)] The topology of $\GM$ is however complicated while $\overline{\T(S)}^{T}$ is a $(6g-6)$-dimenisonal ball.
     \item[(A.4)] Points in the Gardiner-Masur boundary $\PGM$ can be described by homogeneous functions on $\MF(S)$ \cite{Miyachi}. Notice that \cite{Gardiner-Masur}, \cite{Miyachi_II}, $$\mathcal{UE} \varsubsetneqq \PMF(S) \varsubsetneqq \PGM.$$
     The inclusion $\PMF(S) \varsubsetneqq \PGM$ is only set-theoretic while the inclusion $\mathcal{UE} \varsubsetneqq \PGM$ is a topological embedding when $\mathcal{UE}$ is equipped with the induced topology from $\PMF(S)$ \cite{Miyachi_II}. The last proper inclusion also shows that the Gardiner-Masur compactification  is disctinct from the Thurston compactification. 
     \item[(A.5)] The identity map between $\T(S)$ and itself extends to a $\M(S)$-equivariant homeomorphism (\cite{Miyachi_II},\cite{Masur}):$$\T(S) \cup_{T} \mathcal{UE} \cong \T(S) \cup_{GM} \mathcal{UE}. $$ 
     \end{enumerate}

 Because of (A.5), whenever it is possible, the reader is suggested to understand arguments in the sequel using the compactification $\overline{\T(S)}^{T}$.  


\noindent{\textbf{Teichm\"{u}ller geodesics.}}~~Since Teichm\"{u}ller geodesics are important tools for us, we need more precise information about the homeomorphism $\tilde{\Psi}$ from $QD(S)$ to its image in $\MF(S) \times \MF(S)$ \cite{LindenstraussMirzakhani}. The homeomorphism $\tilde{\Psi}$ is defined as follows. Let $\mathcal{C}$ be the set of isotopy classes of simple closed curves on $S$, $i: \MF(S) \times \MF(S) \to \mathds{R}_{+}$ be the intersection function on $\MF(S)$, and $$\tilde{\Delta} = \{(\lambda,\eta) \in (\MF(S))^2: \exists \gamma \in \mathcal{C}, i(\lambda, \gamma) = i(\eta, \gamma) = 0\}.$$ Then $\tilde{\Psi}: QD(S) \to \MF(S) \times \MF(S) \setminus \tilde{\Delta}$ is given by setting $\tilde{\Psi}(q) = (\mathcal{H}(q),\mathcal{V}(q))$.
Let $\Delta = \{(\lambda,[\eta]) \in \MF(S) \times \PMF(S): \exists \gamma \in \mathcal{C}, i(\lambda, \gamma) = i(\eta, \gamma) = 0\}$, then this map further gives a homeomorphism 
\begin{equation}\label{equ:coordinate}
\begin{aligned}
\Psi: QD^{1}(S) \to \MF(S) \times \PMF(S) \setminus \Delta \atop q \mapsto (\mathcal{H}(q), [\mathcal{V}(q)]).
\end{aligned}
\end{equation}
The fact that the map $\Psi$ is a homeomorphism implies that, for $([\lambda],[\eta]) \in \PMF(S) \times \PMF(S) \setminus \Delta$, there is a unique oriented unparameterized Teichm\"{u}ller geodesic with projective horizontal measured foliation $[\lambda]$ and projective vertical measured foliation $[\eta]$. Since every oriented Teichm\"{u}ller geodesic line determines a unique pair $([\xi],[\eta]) \in \PMF(S) \times \PMF(S)$, $\Psi$ gives a parametrization, via a subset of $\PMF(S) \times \PMF(S)$, for the set $\mathfrak{G}$ of all Teichm\"{u}ller geodesic lines on $\T(S)$ in a well-organized manner. This parametrization particularly gives a topological structure, hence a Borel structure, on the set of geodesic lines.

The parametrization of $\mathfrak{G}$ given by $\Psi$ is sufficient and useful for many purposes, but it is not compatible with the Teichm\"{u}ller compactification in the sense that geodesic rays may have many accumulation points in the boundary. Since our approach makes heavily use of geodesics, the Gardiner-Masur compactification  in a way gives a more compatible (but not one to one) parametrization which is sufficient for our purpose. 

It is convenient to use the horofunction compactification to describe this new parametrization of $\mathfrak{G}$. Fix $o \in \T(S)$ and consider the set $C(\T(S),o)$ of all continuous functions on $\T(S)$ with values $0$ at $o$ equipped with the topology of uniform convergence on compact subsets. For any point $x \in \T(S)$, denote $d_x$ the function $y \mapsto d(x,y) -d(x,o)$. Then the assignment $$x \mapsto d_x$$ gives an embedding of $\T(S)$ to $C(\T(S),o)$ with relatively compact closure $\overline{\T(S)}^{horo}$ which is called the \textit{horofunction compactification} of $\T(S)$. The complement $\partial_{h}\overline{\T(S)}$ of $\T(S)$ in $\overline{\T(S)}^{horo}$ is called the \textit{horofunction boundary}. As remarked before, it is known (cf. \cite{Liu-Su_horo}, \cite{Walsh}) that $\GM$ can be identified with $\overline{\T(S)}^{horo}$ in which $\PGM$ is identified with $\partial_{h}\overline{\T(S)}^{horo}$. So we will only talk about $\GM$, and points in $\GM$ will be regarded as functions on $\T(S)$.

Given any geodesic ray $\boldsymbol{r} \subset \T(S)$, it is proved in \cite{Walsh} that, $\boldsymbol{r}$ converges in $\GM$ to a unique point in $\PGM$. In particular, when $\boldsymbol{r}$ is given by $q \in QD^1(S)$ whose vertical measured foliation $\mathcal{V}(q)$ is uniquely ergodic, this unique limit point is given by the image of $\mathcal{V}(q)$ in $\PGM$ under the inclusion $\mathcal{UE} \subset \PGM$. Points in $\PGM$ which are limit points of geodesic rays are called \textit{Busemann points}. So points in $\mathcal{UE} \subset \PGM$ are Busemann points. However, not every point in $\PGM$ is a Busemann point. In fact, there are non-Busemann points in $\PMF(S) \subset \PGM$ \cite{Miyachi}.

For non-Busemann points, we shall use optimal geodesics as asymptotic rays. Recall that, for a horofunction $\xi \in \PGM$, a geodesic $\ell \in \mathfrak{G}$ is called \textit{optimal geodesic for $\xi$} if for all $t \in \mathbb{R}$, $\xi(\ell(t))-\xi(\ell(0)) = -t$. The optimal geodesic for $\xi$ passing through $x$ can be thought as the most efficient way travelling from $x$ to $\xi$. One can check that if $\ell: t \mapsto \ell(t)$ is a parameterized optimal geodesic, then so is $\ell(t+t_0)$ for all $t_0 \in \mathbb{R}$. It is proved in \cite[Proposition 3.4]{azemar2021qualitative} that given $\xi \in \PGM$ and $x \in \T(S)$, the optimal geodesic for $\xi$ passing through $x$ is unique and when $\xi$ is a Busemann point, the unique optimal geodesic passing through $X$ is exactly the geodesic ray defining $\xi$.

To give a slightly more complete picture, we need to extend geodesic rays to geodesic lines. It is shown in \cite[Theorem 1]{Miyachi_intersection} that the intersection number function $i$ on $\MF(S)$ can be extended to the preimage of $\GM$ in $P(\mathbb{R}_{+}^{\mathcal{C}})$. Hence, for any two boundary points $\xi,\eta \in \PGM$, one can talk about whether $i(\xi,\eta)$ vanishes or not. A pair $(\xi,\eta) \in (\PGM)^2 $ is called a \textit{filling pair} if for all $\zeta \in \MF(S) - \{0\}$, $$i(\xi,\zeta) + i(\eta,\zeta)> 0.$$ 
A geodesic line $\ell$ is called \textit{optimal for $(\xi,\eta)$} if $\ell(t)$ is optimal for $\xi$ and $\ell(-t)$ is optimal for $\eta$. 
\begin{lemma}[\cite{Lou-Su-Tan}]
If $(\xi,\eta) \in \PGM \times \PGM$ is a filling pair, then there is a unique optimal geodesic for $(\xi,\eta)$. If both $\xi$ and $\eta$ are furthermore Busemann points, this unique geodesic line converges in the positive direction to $\xi \in \GM$ and in the negative direction to $\eta$.    
\end{lemma}
Notice that if $\xi \in \PMF(S)$ is uniquely ergodic, then for any $\zeta \in \PGM-\{\xi\}$, $(\xi,\zeta)$ is a filling pair \cite[Lemma 7]{Miyachi_intersection}.

If $y \in \T(S)$ and $\xi \in \PGM$, we denote the unique optimal geodesic for $\xi$ passing through $y$ by $\mathcal{L}(y,\xi)$. If $(\xi,\eta) \in (\PGM)^2$ is a filling pair, we denote the unique optimal geodesic for $(\xi,\eta)$ by $\mathcal{L}(\xi,\eta)$. Denote $\mathcal{N}$ the set of non-filling pair. Then there is a map 
\begin{equation}
    \begin{aligned}
 \Upsilon: (\PGM)^2 \setminus \mathcal{N} \to \mathfrak{G}\atop (\xi,\eta) \mapsto \mathcal{L}(\xi,\eta).
  \end{aligned}
\end{equation}
From the above discussions, we infer that $\Upsilon$ is surjective. Clearly, it is not injective. Moreover, the restriction $$\Upsilon|_{\mathcal{UE} \times \mathcal{UE}}$$ is the same as the parametrization given by (\ref{equ:coordinate}). The following lemma shows that, compared to other compactifications, $\GM$ is more compactible with the topological structure of $\mathcal{G}$ when it is equipped with the topology of uniform convergence on compact sets. 
\begin{lemma}\cite{Lou-Su-Tan}\label{lemma:convergence}
    Let $\xi \in \PGM$ be uniquely ergodic and $\eta \in \PGM$ arbitrary. Let $x_n,y_n \in \T(S)$ such that $x_n \to \xi$ and $y_n \to \eta$. Let $\ell_n$ be the geodesic segment passing through $x_n, y_n$. Then the sequence of geodesic segments $\ell_n$ converges to $\mathcal{L}(\xi,\eta)$ uniformly on compact sets. Moreover, if $x_n = \xi$ and $\ell_n$ is the geodesic ray $\mathcal{L}(y_n,\xi)$, then the sequence $\{\ell_n\}$ converges to $\mathcal{L}(\eta,\xi)$.
    
    \end{lemma}
\begin{proof}
    The first part is a consequence of \cite[Theorem 1.5]{Lou-Su-Tan} together with the fact that $\xi$ is uniquely ergodic. We now prove the second part. Fix $t_0$ to be any positive real number and consider $\xi$ as a horofunction. Now consider the horosphere $S(t_0)$ at level $t_0$, i.e. $S(t_0)=\xi^{-1}(t_0)$. We can assume that $\xi(y_n) > t_0$ for all $n$. Since $\ell_n$ and $\mathcal{L}(\eta,\xi)$ are optimal, these rays will have (unique) intersection points, denoted by $p_n,p,$ respectively, with $S(t_0)$. Take a connected segment $I \subset \mathcal{L}(\eta,\xi)$ of length $L$. For any $n$, take $z_n \to \xi$ so that $\xi(z_n) < t_0$ and $d(z_n,p_n) > 2L$. The by the first part, the sequence of geodesic segments $\{[z_n, y_n]\}$ converges to $\mathcal{L}(\xi,\eta)$. This then implies that the sequence $\{\ell_n\}$ converges to $\mathcal{L}(\eta,\xi)$.
\end{proof}

\noindent{\bf Convention:} When $g = 2$, $\M(S_2)$ does not act faithfully on $\T(S_2)$. Instead, if we denote the center of $\M(S_2)$ by $Z$, then the quotient group $\rm{PMod(S_2)}= \M(S_2)/Z$ acts on $\T(S_2)$ faithfully. Since we only care about the orbit for the action of $\M(S_2)$ on $\T(S_2)$, we make the following convention. When we say non-elementary subgroup of $\M(S_2)$, we mean the non-elementary subgroup of $\rm{PMod(S_2)}$.

\section{The Bowen-Margulis measure}
\subsection{Conformal densities}
Let $G$ be a subgroup of $\M(S)$. The subgroup $G$ is called {\it non-elementary} if it contains two independent pseudo-Anosov mapping classes, i.e. two pseudo-Anosov mapping classes with disjoint fixed point sets in $\mathcal{UE} = \mathcal{UE}(S)$, the subset of uniquely ergodic points in $\PMF(S)$. It is also called \textit{sufficiently large} in \cite{McCPapa_Dynamics}. The {\it limit set}, denoted by $\Lambda G$, is defined to be the closure in $\PGM$ of the set of fixed points in $\mathcal{UE} \subset \PGM$ of pseudo-Anosov mapping classes contained in $G$. So if $G$ is non-elementary, then $|\Lambda G| = \infty$. Throughout this paper, the subgroup $G$ is always assumed to be non-elementary. 

One can define a $G-$conformal density in a very general situation, see for instance, \cite{coulon2022patterson} and \cite{yang2022conformal}. Here we only use a simplified version which is sufficient for our purpose. We first recall the Busemann cocycles with respect to the Teichm\"{u}ller metric. Notice that Busemann cocycles here are defined only on $\mathcal{UE}$. Let $x, y \in \T(S)$ and $\xi \in \mathcal{UE}$. The {\it Busemann cocycle $\beta_{\xi}(x,y)$ at $\xi$} is defined to be $$\beta_{\xi}(x,y) = \lim_{z_n \to \xi}(d(x,z_n)-d(y,z_n)),$$ where $z_n \in \T(S)$ and $z_n \to \xi$ in $\GM$. By \cite{Miyachi} or \cite{Walsh}, for $\xi \in \mathcal{UE}$, 
\begin{equation}\label{equ:busemann}
\begin{aligned}
 \beta_{\xi}(x,y) = \frac{1}{2}\ln \frac{\E_{x}(\xi)}{\E_{y}(\xi)}.
\end{aligned}
\end{equation}
Hence $\beta_{\cdot}(x,y)$ is a continuous funtion on $\mathcal{UE}$.
\begin{definition}(Cf. \cite{yang2022conformal} or \cite{coulon2022patterson})
    Let $\mathcal{A} = \{\mu_x: x \in \T(S)\}$ be a family of finite Borel measures on $\PGM$ and $0 \leq \alpha_G \leq \infty$ be a real number. The family $\mathcal{A}$ is called a $\alpha_G-$ dimensional conformal density for a subgroup $G \leq \M(S)$, if it satisfies 
\begin{itemize}
    \item For any $x \in \T(S)$, $\mu_x$ is supported on $\Lambda G$.
    \item For any $x \in \T(S)$ and $g \in G$, $g_{*}\mu_x = \mu_{gx}$.
    \item For any $\xi \in \mathcal{UE} \cap \Lambda G$ and $x,y \in \T(S)$, $$\frac{d\mu_x}{d\mu_y}(\xi) = \exp{(-\alpha_G \beta_{\xi}(x,y))}.$$
\end{itemize}
\end{definition}
Now we recall some general results concerning conformal density for subgroups of mapping class groups from \cite{yang2022conformal}. Notice that $\GM$ is the horofunction compactification of $\T(S)$. Let $s \geq 0$ and $x,y \in \T(S)$. The Poincar\'e series associated to $G \curvearrowright \T(S)$ is $$P_G(s,x,y) = \sum_{g \in G}\exp{(-sd(gx,y))}.$$
\begin{theorem}[Cf.\cite{yang2022conformal}]\label{Yang2022}
    Let $G$ be a non-elementary subgroup of $\M(S)$ and $\mathcal{A} = \{\mu_x\}$ be a positive $\alpha_G-$conformal density for $G$. Then the following two are equivalent:
    \begin{itemize}
        \item[(1)] The Poincar$\acute{e}$ series $P_G(s,o,o)$ diverges at $\alpha_G$ for any (hence for all) $o \in \T(S)$;
        \item[(2)] The set of conical points $(\subset \mathcal{UE})$ has full $\mu_x-$measure.\\
        \end{itemize}
       Each of them implies that the diagonal action of $G$ on $(\PGM)^2$ is ergodic with respect to $\mu_x \times \mu_x$.
        
\end{theorem}

Notice that, since $G$ is non-elementary, any conformal density has positive dimension.
\subsection{The Patterson-Sullivan construction}
The Patterson-Sullivan construction is a recipe for constructing conformal densities. We now briefly review the Patterson-Sullivan construction and the Bowen-Margulis measure in the setting of Teichm\"{u}ller space. We refer to \cite{Quint} for an excellent exposition of the Patterson-Sullivan theory in the setting of negatively curved manifolds.

Let $\delta_G$ be the critical exponent of $G$ with respect to the Teichm\"{u}ller metric $d$. That is, let $o \in \T(S)$,
$$\delta_G = \inf\{s: \sum_{g \in G}\exp(-sd(o,go)) < \infty\}.$$
Notice that $\delta_G$ is independent of the choice of $o$ and $\delta_G > 0$ (since $G$ is non-elementary). Recall that $G$ is said to be {\it of divergence type} if the Poincar$\acute{e}$ series $P_G(s,o,o)$ of $G$ diverges at $\delta_G$. Since $G$ is nonelementary, it has contracting elements \cite{MinskyQuasi}. According to \cite{Yang_scc}, $\delta_G$ is the same as the growth rate of $G$ with respect the Teichm\"{u}ller metric.

We review the Patterson-Sullivan construction by assuming that $G$ is of divergence type. For subgroups of convergence type, one can modify the construction as explained in \cite{Quint}. Let $x \in \T(S)$ and $s > \delta_G$. Denote the Dirac measure at $y \in \T(S)$ by $\mathcal{D}_y$. Consider the family of probability measures $\{\nu_{x,s}\}$ on $\overline{\T(S)}^{Th}$ given by $$ \nu_{x,s} = \frac{1}{P_G(s,x,x)}\sum_{g \in G}\exp(-sd(x,gx))\mathcal{D}_{gx}.$$ Since $\GM$ is compact, the sequence of measures $\{\nu_{x,s}: s > \delta_G\}$ weak$^{*}$ converges to some probability measure $\nu_x$ on $\GM$ as $s \to \delta_G$. We now claim that the support of $\nu_x$ is $\PGM$. Indeed, since $P_G(s,x,x)$ diverges at $\delta_G$ by the assumption and the action of $G$ on $\T(S)$ is proper, $\nu_{x}(K) = 0$ for any compact subset $K$ of $\T(S)$. Hence it is supported on $\PGM$. Now define $\nu_y$ for every  $y \in T(S)$ as $$ \forall \xi \in \mathcal{UE}, d\nu_y(\xi)= \exp(-\delta_G\beta_{\xi}(y,x))d\nu_x(\xi).$$

\begin{lemma}
    If $G$ is of divergence type, then the family of probability measures $\{\nu_x:x \in \T(S)\}$ is a $\delta_G-$conformal density for $G$ which has no atoms. Moreover, each measure $\nu_x$ is fully supported on $\Lambda G \cap \mathcal{UE}$. 
\end{lemma}
\begin{proof}
    Since $G$ is non-elementary, $G$ has pseudo-Anosov mapping classes. By \cite{MinskyQuasi}, pseudo-Anosov mapping classes are contracting elements for the action of $\M(S)$ on $\T(S)$. On the other hand, \cite[Theorem 1.2]{yang2022conformal} shows that the Gardiner-Masur boundary, with finite difference relation, gives a nontrivial convergence boundary. One then can apply \cite[Lemma 6.2]{yang2022conformal} to conclude the fact that $\{\nu_x\}$ is a $\delta_G$-conformal density for $G$. Note that here $\{\nu_x\}$ is a conformal density instead of quasi-conformal density. This follows from the fact that, $\beta_{\xi}(x,y) = \beta_{\xi}(x,z) + \beta_{\xi}(x,y)$ for $\xi \in \mathcal{UE}$, one has $$\forall  \xi \in \mathcal{UE}, z,y \in \T(S), d\nu_z(\xi)= \exp(-\delta_G\beta_{\xi}(z,y))d\nu_y(\xi).$$

      By \cite[Lemma 6.2]{yang2022conformal} and Theorem \ref{Yang2022}, one can see that $\nu_x$ is fully supported on $\Lambda G \cap \mathcal{UE}$. To see that $\nu_x$ has no atoms, suppose $\nu_x$ has an atom $\eta \in \Lambda G$, say of mass $r$. Then by Theorem \ref{Yang2022}, $\eta$ should be a conical point, that is there exists a $D > 0$ such that the $D$ neighborhood of the geodesic $[x, \eta)$ intersects the orbit $Gx$ infinitely many times. Let $\gamma_n \in G$ be such a sequence. Then by the triangle inequality, 
    $\beta_{\eta}(\gamma_n x,x) \to -\infty$. So 
    $$\nu_x(\gamma_n^{-1}\eta) = \nu_{\gamma_n x}(\eta) = \exp(-\delta_G\beta_{\eta}(\gamma_n x,x))\nu_x(\eta) \to \infty $$ which contradicts the finiteness of $\nu_x$.

    \end{proof}

\begin{corollary}
    If $G$ is a non-elementary subgroup of $\M(S)$ of divergence type. Then there exists a $\delta_G-$conformal density $\{\nu_x: x \in \T(S)\}$ without atoms for $G$ which is fully supported on the set of conical points in $\mathcal{UE}$. 
\end{corollary}

Since there is a unique way to extend each $\nu_x$ to be a complete measure, we always assume such measures are complete. We now turn to the definition of the Bowen-Margulis measure on $QD^1(S)/G$. For $x\in \T(S)$ and $\xi, \eta \in \mathcal{UE}$ with $\xi \neq \eta$, let 
\begin{equation*}\label{equ:BowenMargulis}
    \begin{aligned}
    \rho_{x}(\xi,\eta) = \lim_{z \to [\xi], w \to [\eta]} (d(x,z) + d(x,w) - d(z,w))
    &=  \beta_{\xi}(x,u) + \beta_{\eta}(x,u),
    \end{aligned}
\end{equation*}
where $u \in \T(S)$ is any point in the geodesic $(\xi,\eta)$. Now we define a measure $\hat{\mu}_{G}$ on $\PMF(S) \times \PMF(S)$ by setting $$\forall \xi \neq \eta \in \mathcal{UE},~~~ d\hat{\mu}_{G}(\xi,\eta) = \exp{(\delta_G \rho_x(\xi,\eta))}d\nu_x(\xi)d\nu_x(\eta).$$ Notice that since $\nu_x$ supported on $\mathcal{UE}$, $\hat{\mu}_G$ extends to $\PMF(S) \times \PMF(S)$. We claim that $\hat{\mu}_G$ is independent of $x$ and it is $G-$invariant for the diagonal action of $G$ on $\PMF(S) \times \PMF(S)$. In fact this claim follows from the definition of the conformal density and the definition of $\rho_x$. 

The homeomorphism $\Psi$ above gives a \textit{Hopf parametrization} of the unit cotangent bundle $QD^1(S)$ of $\T(S)$ as $(\PMF(S) \times \PMF(S) \setminus \Delta) \times \mathbb{R}$ where $\Delta = \{([\lambda],[\eta]) \in \PMF(S) \times \PMF(S): \exists \gamma \in \mathcal{C}, i(\lambda,\gamma)=i(\eta,\gamma)=0\}$. Therefore each point in $QD^1(S)$ can be parameterized as $(\xi,\eta,s)$. Then we define a measure $\tilde{\mu}_G$ on $QD^1(S)$ by $$d\tilde{\mu}_G = dtd\hat{\mu}_G.$$ Notice that since $\nu_x$ is fully supported on $\mathcal{UE}$, $\tilde{\mu}_G$ is fully supported on geodesic lines with both uniquely ergodic horizontal and vertical measurable foliations. The measure $\mu_G$ on the quotient $QD^1(S)/G$ that descends from $\tilde{\mu}_G$ is called {\it the Bowen-Margulis measure}. A priori, the Bowen-Margulis measure $\mu_G$ may be infinite. 



\section{Mixing of the Bowen-Margulis measure}
In this section, we prove that the geodesic flow $\{g_t\}$ on $QD^1(S)/G$ is mixing for a subgroup $G$ with finite Bowen-Margulis measure. Our argument closely relates to Babillot's \cite{Babillot} for rank one compact manifolds. First, we show that the length spectrum associated to every non-elementary subgroup of $\M(S)$ is nonarithmetic. Then by combining with double ergodicity, we show the mixing of the geodesic flow.
\subsection{Nonarithmeticity of the length spectrum}
For a non-elementary subgroup  $G \leq \M(S)$, denote the spectrum of $G$ by $$Spec(G) = \{\log(\lambda_{g}): g \in G \text{~is pseudo-Anosov and $\lambda_g$ is the dilatation of $g$}\}.$$ It is a closed discrete subset of $\mathbb{R}$ which coincides with lengths of geodesic loops corresponding to elements in $G$ in the moduli space $\T(S)/\M(S)$ with respect to the Teichm\"{u}ller metric.

\begin{theorem}\label{thm:nonarithmetic}
    Let $G$ be a non-elementary subgroup of $\M(S)$. Then the spectrum $Spec(G)$ generates a dense subgroup of $\mathds{R}$.
\end{theorem}
The proof is based on the following lemma proved in \cite[Proposition 2.3]{GuivarUrban}
\begin{lemma}\label{lemma:linardensity}
    Let $\Gamma \subset {\rm GL(m,\mathds{R})}$ be a semigroup acting strongly irreducibly on $\mathds{R}^m$ and containing a proximal element. Then the logarithms of maximal eigenvalues of proximal elements of $\Gamma$ generate a dense subgroup of $\mathds{R}$.
\end{lemma}
To apply this lemma, we will use the piecewise linear structure on $\MF(S)$, see \cite{PennerHarer} for more details for piecewise linear structure, train tracks and etc. 

Given any train track $\tau$, let $W_{\tau}$ be the vector space of weights on the branches of $\tau$ satisfying the switch condition. Let $V_{\tau} \subset W_{\tau}$ be the open cone consisting of positive weights at each branch of $\tau$. One can assign a measured foliation to a vector in $V_{\tau}$. Denote this assignment by $\phi_{\tau}: V_{\tau} \to \MF(S)$. Define $$U_{\tau}= \phi_{\tau}(V_{\tau}),~~
\psi_{\tau}= \phi^{-1}_{\tau}: U_{\tau} \to V_{\tau}.$$ 
The following lemma is well-known (cf. \cite{Penner}).
\begin{lemma}
    Let $\gamma \in \M(S)$ be a pseudo-Anosov mapping class. Then one can choose $\tau$ carefully such that it is a $\gamma$-invariant maximal trivalent birecurrent train track and $\gamma$ acts linearly on $W_{\tau}$ by a matrix with nonnegative entries whose largest eigenvalue is the dilatation $\lambda_{\gamma}$ of $\gamma$.
\end{lemma}
In fact, one could also choose $\tau$ so that $\tau$ carries $\mathfrak{F}^{u}(\gamma)$ but does not carry $\mathfrak{F}^{s}(\gamma)$.

Let $G$ be a nonelementary subgroup of $\M(S)$. Let $\gamma_1 \in G$ be a pseudo-Anosov mapping class and $\tau$ be a $\gamma_1$-invariant trivalent maximal birecurrent train track given by the above lemma. Then $\mathfrak{F}^{u}(\gamma_1)$ is carried by $\tau$. Choose $\gamma_2 \in G$ such that $\tau$ is also $\gamma_2$-invariant and satisfies the above lemma. After replacing $\gamma_i$ by a large enough power, on can assume that $\gamma_i$ acts on $W_{\tau}$ via positive matrices $M(\gamma_i)$ whose largest eigenvalue is exactly $\lambda_{\gamma_i}$ and each $\gamma_i$ preserves $U_{\tau}$. We remark that it follows from our assumption that each $M(\gamma_i)$ is proximal (cf. \cite{Benoist-Quint}). Moreover, each $M(\gamma_i)$ is symplectic since $\M(S)$ acts on $\MF(S)$ preserving the Thurston symplectic form. Let $\Gamma = \langle M(\gamma_1),M(\gamma_2)\rangle$ be semigroup generated by $M(\gamma_1)$ and $M(\gamma_2)$.

Now, let $\gamma \in G$ be a pseudo-Anosov mapping class and $A = M(\gamma)$ be the corresponding matrix associated to a trivalent, maximal birecurrent train track $\tau$ which does not carry $\mathfrak{F}^{s}(\gamma)$. Suppose $\gamma$ preserves $U_{\tau}$. Denote $v_A= \psi_{\tau}(\mathfrak{F}^{u}(\gamma))$ and $E_A$ the direct sum of complementary eigenspaces of $v_A$. Then we claim that $E_A \cap V_{\tau} = \emptyset$, i.e. $E_A$ cannot have any vector with all entries nonnegative. Indeed, if $v \in E_A$ then on the one hand, $$[A^nv] \nrightarrow [v_A].$$ On the other hand, since $\mathfrak{F}^{s}(\gamma) \notin U_{\tau}$, any $v \in V_{\tau}$ has the property that $[A^nv]$ converges to $[v_A]$ as $n \to \infty$. As a consequence, for $i \neq j$, $\Gamma v_{M(\gamma_i)} \cap E_{M(\gamma_j)} = \emptyset$.

For a proximal element $A \in {\rm SL(n,\mathds{R})}$, let $v_A$ be the unit vector corresponding to the unique eigenvalue of maximal absolute value $\lambda_A$ (cf. \cite{Benoist-Quint}) and $E_A$ be the direct sum of complementary eigenspaces of $v_A$. We call $v_A$ a dominated eigenvector of $A$. Hence the proof of Theorem \ref{thm:nonarithmetic} is then reduced to the following theorem about linear semigroup actions on the projective spaces.
\begin{theorem}\label{thm:linearnonarithematic}
    Let $\Gamma$ be a semigroup of ${\rm SL(n,\mathds{R})}$ which contains a subsemigroup $\Gamma'$ whose all elements are proximal. If for every $A,B \in \Gamma'$, one has $\Gamma' v_B \cap E_A = \emptyset$. Then the logarithms of maximal eigenvalues of matrices in $\Gamma$ generate a dense subrgoup of $\mathds{R}$.
\end{theorem}
\begin{proof}
    Without loss of generality, we can assume $\Gamma = \Gamma'$. The proof is complete by combining Lemma \ref{lemma:stronglyirreducibly} and Lemma \ref{lemma:linardensity}.
\end{proof}
\begin{proof}[Proof of Theorem \ref{thm:nonarithmetic}] 
The proof is completed by Theorem \ref{thm:linearnonarithematic} and the reduction stated before.   
\end{proof}

\begin{lemma}\label{lemma:stronglyirreducibly}
    Under the assumptions of Theorem \ref{thm:linearnonarithematic}, for independent $A,B \in \Gamma' \subset \Gamma$, there is an integer $M > 0$ such that the semigroup $\Gamma_M = \langle A^M, B^M \rangle$ acts strongly irreducibly on some vector space $V_{\Gamma_M}$.
\end{lemma}

In the rest of this section we devote to prove Lemma \ref{lemma:stronglyirreducibly}. All arguments in the rest of this subsection are under the assumption of Theorem \ref{thm:linearnonarithematic}. Let $\Gamma$ be a semigroup consisting of proximal elements. Define the limit set $L_{\Gamma}$ of $\Gamma$ to be the closure in $\mathds{R}P^{n-1}$ of $$L_{\Gamma} = \{[v_A]: A \in \Gamma\}.$$
Denote by $\widetilde{L_{\Gamma}}$ the preimage of $L_{\Gamma}$ in $\mathds{R}^n$.
\begin{lemma}\label{lemma:Gammainvariance}
    The limit set $L_{\Gamma}$ is $\Gamma$-invariant.
\end{lemma}
\begin{proof}
    Suppose $A,B \in \Gamma$, we want to show that $[Av_{B}] \in L_{\Gamma}$. Consider $$u= \lim_{k \to \infty}\frac{1}{\lVert B\rVert}B^k,$$ where $\lVert B\rVert$ is the operator norm of $B$. Then $u$ is a projection onto $\mathds{R}v_B$ with kernel $E_B$ and $u(v_B)=v_B$. By assumption, $Av_B \notin E_B$ which is ${\rm Ker}u = {\rm Ker}Au$. Hence $Au$ is a multiple of a projection onto $\mathds{R}Av_B$. Now $$Au =\lim_{n \to \infty}\frac{1}{\lVert B\rVert}AB^n,$$ so $[v_{AB^n}] \to [Av_{B}]$. Therefore $[Av_B] \in L_{\Gamma}$.
\end{proof}
\begin{lemma}\label{lemma:dichotomy}
    Any $\Gamma$-invariant subspace $W$ either belongs to $\bigcap_{A \in \Gamma}E_A$ or contains $v_{A}$ for all $A \in \Gamma$.
\end{lemma}
\begin{proof}
    Let $\Gamma W = W$. Suppose that there is $B_0 \in \Gamma$ but $v_{B_0} \notin W$. We claim that $W \subset E_{B_0}$. Indeed otherwise for every $v \in W-E_{B_0}$ we have $$\lim_{n\to \infty}[B_0^nv] = [v_{B_0}] \notin [W].$$ However for any $v \in W$, any limit point of $\{[B_0^kv]: k \in \mathds{N}\}$ is in $W$ since $W$ is closed $\Gamma$-invariant subset, a contradiction.  Since by the assumptions of Theorem \ref{thm:linearnonarithematic}, for any $A \in \Gamma$, we have $v_A \notin E_{B_0}$ which in particular implies $v_A \notin W$. Thus by the same argument, $W \subset\bigcap_{A \in \Gamma}E_A$.  
\end{proof}
Let $W_{\Gamma}$ be the smallest subspace of $\mathds{R}^n$ containing $\widetilde{L_{\Gamma}}$, that is $$W_{\Gamma} = \bigoplus_{A \in \Gamma}\mathds{R}v_{A}.$$
By Lemma \ref{lemma:Gammainvariance}, $L_{\Gamma}$ is $\Gamma$-invariant, hence $W_{\Gamma}$ is also $\Gamma$-invariant. Let $U_{\Gamma}$ be a maximal proper $\Gamma$-invariant subspace of $W_{\Gamma}$. Denote $V_{\Gamma} = W_{\Gamma}/U_{\Gamma}$.
\begin{lemma}\label{lemma:irreducibility}
    $\Gamma$ acts irreducibly on $V_{\Gamma}$ and each $A \in \Gamma$ has the same largest eigenvalues in this action as in the action on $\mathds{R}^n$.
\end{lemma}
\begin{proof}
    By maximality of $U_{\Gamma}$, one has irreducibility. Since $U_{\tau}$ does not contain $v_A$ for all $A \in \Gamma$, by Lemma \ref{lemma:dichotomy}, $U_{\Gamma} \subset \bigcap_{A \in \Gamma}E_A$ and so does not contain any $v_A$ for $A \in \Gamma$ so $v_A + U_{\Gamma}$ is a dominant eigenvector for $A$ with eigenvalue $\lambda_{A}$.
\end{proof}
\begin{proof}[Proof of Lemma \ref{lemma:stronglyirreducibly}]
   Consider independent $A, B \in \Gamma$ and for each $k$, consider vector spaces $$W_{\Gamma_k},U_{\Gamma_k},V_{\Gamma_k} =W_{\Gamma_k}/U_{\Gamma_k}.$$ Since all vectors involved are finite dimensional, there is an $N > 0$ and subspaces $U,W,V$ such that $$W=W_{\Gamma_M}, U=U_{\Gamma_M}, V=V_{\Gamma_M}=W/U,$$ for all $M \geq N$. Fix such $N$ and consider $\Gamma_N$, the semigroup generated by $A^N, B^N$, and $G_N$, the group generated by $A^N,B^N$. By Lemma \ref{lemma:irreducibility}, $\Gamma_N$ acts irreducibly on $V$ and hence so does $G_N$. Now suppose that $\Gamma_N$ does not act strongly irreducibly on $V$. Take $V_1,\dots, V_m \subset V$ which is a minimal collection of subspaces of $V$ such that their union is preserved by $\Gamma_{N}$. Since $A^N, B^N$ preserves the collection $\{V_i\}$, so there is an integer $d > 0$ such that $A^{Nd}V_i = V_i, B^{Nd}V_i=V_i$ for each $i$, So the semigroup $\Gamma_{Nd}$ generated by these two matrices preserves the nontrivial proper subspace $V_1 \subset V$ contradicting Lemma \ref{lemma:irreducibility}.
  \end{proof}

\subsection{Mixing of the Bowen-Margulis measure}
We are now in a position to show the mixing of the geodesic flow with respect to the Bowen-Margulis measure $\mu_G$. Notice that, by the Poincar$\acute{e}$ recurrence theorem and Theorem \ref{Yang2022}, finiteness of $\mu_G$ implies the divergence of $G$.

\begin{theorem}\label{theorem:mixing}
    Let $G$ be a non-elementary subgroup of $\M(S)$ with finite Bowen-Margulis measure $\mu_G$, then the geodesic flow $g_t$ on $QD^1(S)/G$ is mixing with respect to $\mu_G$. More precisely, let $A, B \subset QD^1(S)/G$ be Borel sets with finite $\mu_G-$measures, if $\mu_G(QD^1(S)/G) < \infty$, then $$\lim_{t \to \pm \infty}\mu_G(A \cap g_{-t}B) = \frac{\mu_G(A)\mu_G(B)}{\mu_G(QD^1(S)/G)}.$$ Otherwise,  $$\lim_{t \to \pm \infty}\mu_G(A \cap g_{-t}B) = 0.$$
\end{theorem} 
Since the proof is very long, we first state the following consequence.
\begin{corollary}\label{coro:mixing}
Let $G$ be a non-elementary subgroup of $\M(S)$ with finite Bowen-Margulis measure $\mu_G$. Let $A, B \subset QD^1(S)/G$ be any Borel subsets with finite $\mu_G-$measures such that the quotient map $QD^1(S) \to QD^1(S)/G$ maps injectively $\tilde{A}, \tilde{B}$ to $A,B$, respectively. If $\mu_G(QD^1(S)/G) < \infty$, then $$\lim_{t \to \pm \infty}\sum_{\gamma \in G}\tilde{\mu}_G(\tilde{A} \cap g_{-t}\gamma\tilde{B}) = \frac{\tilde{\mu}_G(A)\tilde{\mu}_G(B)}{\mu_G(QD^1(S)/G)}.$$ Otherwise, $$\lim_{t \to \pm \infty}\sum_{\gamma \in G}\tilde{\mu}_G(\tilde{A} \cap g_{-t}\gamma\tilde{B}) = 0.$$
\end{corollary}
One can see easily that, if Corollary \ref{coro:mixing} holds for a finite index subgroup of $G$, it holds for $G$. Since $\M(S)$ is virtually torsion free, we can thus assume that $G$ is torsion free. The proof of Corollary \ref{coro:mixing} is the same as \cite[Corollary 5.6]{Link}. The only nontrivial ingredient used in that proof, when we unwrap the integral, is the existence of measurable fundamental domain with null measure boundary \cite[Lemma 4.4]{Link}. So we need to provide the following analogue.

Recall that an open subset $F$ of $Y \subset QD^1(S)$ is called a {\it fundamental domain} for the action of $G$ on $Y$  if $Y = \bigcup_g g\bar{F}$ and $gF \cap F = \emptyset$ for all $g \neq id$ in $G$. Furthermore, it is called {\it locally finite} if the collection $\{gF: g \in G\}$ is a locally finite collection. Notice that $\M(S_2)$ is understood as $\rm{PMod(S_2)}$.
\begin{lemma}
    Let $G$ be a torsion-free non-elementary subgroup of $\M(S)$ and $\tilde{\mu}_G$ be the measure on $QD^1(S)$ which decends to the Bowen-Margulis measure on $QD^1(S)/G$. Then there is a $G-$invariant measurable subset $Y \subset QD^1(S)$ with full $\tilde{\mu}_G-$measure. Moreover, there also exists a locally finite fundamental domain $F$ for the action of $G$ on $Y$ so that $\tilde{\mu}_G(\partial F) = 0$. 
\end{lemma}
\begin{proof}[Sketch of the proof]
  The proof is essentially in \cite[Lemma 4.4]{Link}, but here we add some comments since for subgroups of $\M(S)$, the argument is more concrete. The idea is to use the fundamental domain constructed in \cite{McCarthyPapad} for the action of $G$ on $\T(S)$, then use the trick in \cite{Roblin} to modify the boundary to have null measure as explained in \cite[Lemma 4.4]{Link}.
  
  First consider $$\mathcal{G} = Fix(G)= \{q \in QD^1(S): \exists \gamma \in G \setminus \{id\}, \gamma q = q\}.$$
  Then $\mathcal{G}$ is closed, $G-$invariant and $g_t-$invariant. Notice that, if $\gamma q = q$, then $\gamma (\pi(q)) = \pi(q)$ where $\pi: QD^1(S) \to \T(S)$ is the projection. Such $\gamma$ should be of finite order. Since $G$ is torsion-free, $\mathcal{G}$ is thus empty.  
  
  Now choose a point $m_0 \in \T(S)$ with trivial stabilizer in $G$ and consider the Dirichlet domain associated to the action of $G$ on $\T(S)$, namely, $$\mathds{D}_G = \{y \in \T(S): d(y,m_0) < d(y,\gamma m_0), \forall \gamma \in G \setminus {id}\}.$$
  By \cite[Theorem 5.9]{McCarthyPapad}, $\mathds{D}_G$ is a locally finite fundamental domain for the action of $G$ on $\T(S)$. In fact the boundary $\partial \mathds{D}_G$ is a properly embedded submanifold of codimension one of $\T(S)$. Now consider the preimage $\mathcal{D}_G$ of $\mathds{D}_G$ in $QD^1(S)$ under the projection $\pi: QD^1(S) \to \T(S)$. That is,
  $$\mathcal{D}_G = \{q \in QD^1(S): d(\pi(q),m_0) < d(\pi(q),\gamma m_0), \forall \gamma \in G \setminus {id}\}.$$ Then we know that $\mathcal{D}_G$ is a fundamental domain for the action of $G$ on $QD^1(S)$ and $\pi(\mathcal{G}) \subset \partial \mathds{D}_G$. Notice that, in our construction $\bar{\mathcal{D}}_G = \pi^{-1}(\bar{\mathds{D}}_G)$, so unlike in \cite{Link}, it is not necessary to consider the so-called enlarged boundary. Nevertheless, it may happen that $\partial \mathcal{D}_G$ has positive $\mu_G-$measure, so we shall follow the same argument in \cite[Lemma 4.4]{Link} to conclude the existence of a required fundamental domain. We only remark here that, one can replace the metric involved in \cite[Lemma 4.4]{Link} by the $\M(S)-$invariant distance on $QD^1(S)$, defined by (see \cite{yang2022conformal}) $$\tilde{d}(q_1,q_2) = \int^{+\infty}_{-\infty}\frac{d(\pi(g_tq_1),\pi(g_tq_2))}{2e^t}dt.$$
\end{proof}
The rest of this subsection is devoted to the proof of Theorem \ref{theorem:mixing}.

The proof of Theorem \ref{theorem:mixing} is modelled on Babillot's argument \cite{Babillot} (see also \cite[Theorem 5.4]{Link}). First we recall a lemma (Cf. \cite[Lemma 1]{Babillot}) from the unitary representation theory.
\begin{lemma}\label{lemma:unitaryrepn}
Let $(X, \mathcal{B},m, (T_t)_{t\in A})$ be a measure preserving dynamical system, where $(X,\mathcal{B})$ is a standard Borel space, $m$ a possibly unbounded Borel measure on $(X,\mathcal{B})$ and $(T_t)_{t \in A}$ an action of a locally compact second countable abelian group $A$ on $X$ by measure preserving transformations. Let $\phi \in L^2(X,m)$ be a real-valued function on $X$ such that $\int \phi dm=0$ if $m$ is finite. If there exists a sequence $(t_n)$ going to infinity in $A$ such that $\phi \circ T_{t_n}$ does not converge to $0$ in the weak $L^2$ topology, then there exist a sequence $(s_n)$ going to infinity in $A$ and a non-constant function $\chi$ in $L^2(X,m)$ such that, in the weak $L^2$ topology, $$\phi \circ T_{s_n} \to \chi,~~ \phi \circ T_{-s_n} \to \chi.$$ 
\end{lemma}

\begin{proof}[Proof of Theorem \ref{theorem:mixing}]
    Denote $X=QD^1(S)/G$, $\mu = \mu_G$ and $$\|\mu\| = \mu(QD^1(S)/G).$$ To show mixing, we only need to show that, for every $f \in C_c(X)$ (with $\int f d\mu = 0$ if $m$ is finite), $$f \circ g_t \to 0 $$ weakly in $L^2$ at $t \to \pm \infty$.

    Suppose that $\mu$ is not mixing. Then there exist $f \in C_c(X)$( with $\int f d\mu = 0$ if $m$ is finite) and $t_n$ going to infinity such that $f \circ g_{t_n}$ does not weakly in $L^2$ converge to $0$. By Lemma \ref{lemma:unitaryrepn}, there is a non-constant function $\chi \in L^2(X,\mu)$ and a sequence $s_n$ going to $\infty$, such that, weakly in $L^2(X,\mu)$, $$f \circ g_{s_n} \to \chi,~~ f \circ g_{-s_n} \to \chi ~~(n \to \infty).$$ By \cite[Fact]{Babillot}, $\chi$ is also the almost sure limit of Cesaro average of $f$ for positive and negative times.
    
    Modify $\chi$ so that $\chi$ is defined on $X$ and denote $\Xi$ the lift of $\chi$ to the unit tangent bundle $Y= QD^1(S)$ on $\T(S)$. For each $\epsilon > 0$, define the $\epsilon$-smooth $\Xi_{\epsilon}$ to be $$\Xi_{\epsilon}: q \in Y \mapsto \int^{\epsilon}_{0}\Xi(g_s(q))ds,$$here $g_t$ is the geodesic flow on $Y$. Since $\chi$ is not constant, by taking $\epsilon$ small enough one can ensure that $\Xi_{\epsilon}$ is not constant. Now there is a full $\tilde{\mu}$-measure subset $E_0 \subset \Lambda G \times \Lambda G$ such that for any $q \in Y$ with $([\mathcal{H}(q)],[\mathcal{V}(q)]) \in E_0$, the function $$h_q: t \mapsto \Xi_{\epsilon}(g_tq)$$ is well defined and continuous for any real number $t$. Here, since $\nu_x$ is supported on $\mathcal{UE} \cap \Lambda G$, after modifying a null measure set, we regard $\Lambda G \times \Lambda G$ as a subset of $\PMF(S) \times \PMF(S)$ by considering the Thurston compactification $\overline{\T(S)}^{Th}$. For each $q \in Y$, consider the period ${\rm Per(q)}$ of the map $h_q$. We note that ${\rm Per(q)}$ might be trivial and it only depends on the geodesic containing $q$. Thus it gives us a measurable map $${\rm Per}: E_0 \to \rm{Sub}$$ where $\rm{Sub}$ is the set of closed subgroups of $\mathds{R}$. Note that ${\rm Per}$ is also $G$-invariant since $\Xi_{\epsilon}$ is $G$-invariant. By the double ergodicity with respect to $\nu_x \times \nu_x$  (Theorem \ref{Yang2022}), ${\rm Per}$ is essentially constant. 

    Suppose that this constant is $\mathds{R}$. Then $\Xi_{\epsilon}$ is $g_t$-invariant and hence it defines a $G$-invariant measurable function on $\Lambda G \times \Lambda G - {\rm Diag}$. By ergodicity of $\tilde{\mu}$, $\Xi_{\epsilon}$ is constant which contradicts our choice of $\epsilon$. So there exists $k \geq 0$ such that almost surely ${\rm Per}$ is $k\mathds{Z}$. Denote this full measure subset of $E_0$ by $E_1$.

    Our goal is to show that, in this case, $Spec(G)$ is contained in $\frac{1}{2}k\mathds{Z}$ contradicting Theorem \ref{thm:nonarithmetic}. 
    
    We first choose a full measure subset of $E$ with nice properties. Consider the lift $\tilde{f}$ of $f$ to $Y$ and define $$\tilde{f}_{\epsilon}: Y \to \mathcal{R}, q \mapsto \int^{\epsilon}_{0}\tilde{f}(g_sq)ds.$$
    By $G$-invariance of $\tilde{f}$, one gets a function $f_{\epsilon}$ on $X$. By Lemma \ref{lemma:unitaryrepn} and \cite[Fact]{Babillot}, one has a non-constant function $\chi_{\epsilon}$ as a weakly limit in $L^2(X,\mu)$ which is the push forward function of $\Xi_{\epsilon}$. It is also the almost sure limit of the Cesaro averages for positive and negative times  $$\frac{1}{K^2}\sum^{K^2}_{k=1}f_{\epsilon}\circ g_{\pm s_{n_{k}}}$$ for some sequence $\{n_k\}$. Denote $\Xi^{\pm}_{\epsilon}$ the lifts of the limit supreme of the Cesaro averages defined above. By Fubini's theorem, there is a set $E_2 \subset E_1$ of full measure and such that for every $q \in E_2$, $$\Xi_{\epsilon}^{+}(q)=\Xi_{\epsilon}^{-}(q)=\Xi_{\epsilon}(q)$$ and $E_2 \subset \mathcal{UE} \times \mathcal{UE}$. Furthermore, define $E^{-}$ to be the set of $[\zeta] \in \Lambda G$ such that for $\nu_x$-almost sure $[\eta]$, $([\zeta],[\eta]) \in E_2$ and $E^{+}$ to be the set of $[\zeta] \in \Lambda G$ such that for $\nu_x$-almost sure $[\eta]$, $([\eta],[\zeta]) \in E_2$. By Fubini's theorem, $E = E_2 \bigcap (E^{+} \times E^{-})$ has full measure.

   We now introduce the \textit{cross ratio} of four points in $E$ which is in fact defined for four points in $\mathcal{UE}$ and relate it to $\log (\lambda_{g})$ for a pseudo-Anosov mapping class $g \in G$. 
   
   Take $([\zeta_1],[\eta_1]) \in E$, since $E$ has full measure, one then can choose $([\zeta_2],[\eta_2]) \in E_2$ so that $([\zeta_1],[\eta_2]) \in E_2$ and $([\zeta_2],[\eta_1]) \in E_2$. Now take $q_0 \in QD^1(S)$ so that $([\mathcal{V}(q_0)],[\mathcal{H}(q_0)]) = ([\zeta_1],[\eta_1])$ and define $p_0 = \pi(q_0)$ where $\pi: QD^1(S) \to \T(S)$ is the projection. The pair $(q_0,p_0)$ is said to be in the geodesic $([\zeta_1],[\eta_1])$. Once $(q_0,p_0)$ has been fixed, there is unique pair $(q_1,p_1)$ in the geodesic $([\zeta_1],[\eta_2])$ so that $\E_{p_0}(\mathcal{V}(q_0))=\E_{p_1}(\mathcal{V}(q_0))$. Consider $\mathcal{H}(q_1)$, there is a unique pair $(q_2,p_2)$ in the geodesic $([\eta_2],[\zeta_2])$ so that $\E_{p_1}(\mathcal{H}(q_1)) =\E_{p_2}( \mathcal{H}(q_1))$. Continue this process, one finds a pair $(q_3,p_3)$ in $([\zeta_2],[\eta_1])$ so that $\E_{p_2}(\mathcal{V}(q_2))=\E_{p_3}(\mathcal{V}(q_2))$ and a pair $(q_4,p_4)$ in $([\zeta_1],[\eta_1])$ with $\E_{p_3}(\mathcal{H}(q_3)) =\E_{p_4}( \mathcal{H}(q_3))$. Now the two pairs $(q_0,p_0)$ and $(q_4,p_4)$ are in the same geodesic $([\zeta_1],[\eta_1])$, there is a unique real number $\iota > 0$ such that $$g_{\pm \iota}q_{0} = q_4.$$
We call this number $\iota$ the \textit{cross ratio} of $([\zeta_1],[\eta_1],[\zeta_2],[\eta_2])$ and denote it by $$\iota = ([\zeta_1],[\eta_1];[\zeta_2],[\eta_2]).$$ 

We collect several properties of the cross ratio defined above. First, by a result of Gardiner-Masur(\cite{Gardiner-Masur}), if $q \in QD^1(X)$ with $(\zeta,\eta) = (\mathcal{V}(q),\mathcal{H}(q))$, then $$i^2(\zeta,\eta)= \E_X(\zeta)\E_X(\eta).$$ This gives us that $$([\zeta_1],[\eta_1];[\zeta_2],[\eta_2]) =| \ln \frac{i(\eta_1,\zeta_1)i(\eta_2,\zeta_2)}{i(\eta_1,\zeta_2)i(\eta_2,\zeta_1)}|.$$ 
Hence the cross ratio depends only on the choice of $([\zeta_1],[\eta_1],[\zeta_2],[\eta_2]) \in \mathcal{UE}^4$. One can also deduce from this formula that it can be extend to all point in $\mathcal{UE}^4$ and it is continuous. Moreover, for any pseudo-Anosov mapping class $g \in \M(S)$, by considering $\mathcal{F}^{u}(g), \mathcal{F}^{s}(g) \text{~and~} [\beta],g[\beta]$ for any $[\beta]\in \mathcal{UE}$, one has that the cross ratio $$([\mathcal{F}^{u}(g)], [\beta]; [g\beta],[\mathcal{F}^{u}(g)]) = 2\ln(\lambda_g).$$ 

We now claim that, if we start from $([\zeta_1],[\eta_1]) \in E$ and choose $[\zeta_2],[\eta_2]$ as above, then $([\zeta_1],[\eta_1];[\zeta_2],[\eta_2])$ is in $k\mathds{Z}$, the closed subgroup of $\mathds{R}$ mentioned before. By combining with the previous observation about relationship between cross ratio and the spectrum of $G$, one can conclude the proof.

For the claim, indeed, by \cite[Theorem 1.1]{Walsh}, one can see that $$\lim_{t \to +\infty}d_{T}(\pi(g_tq_0),\pi(g_tq_1)) = 0,\lim_{t \to -\infty}d_{T}(\pi(g_tq_1),\pi(g_tq_2)) = 0,$$  $$
\lim_{t \to +\infty}d_{T}(\pi(g_tq_2),\pi(g_tq_3)) = 0,\lim_{t \to -\infty}d_{T}(\pi(g_tq_3),\pi(g_tq_4)) = 0.$$ So, by the definition of $E$, 
\begin{equation}
    \begin{aligned}
&\Xi_{\epsilon}(q_0)=\Xi^{+}_{\epsilon}(q_0)=\Xi^{+}_{\epsilon}(q_1)= \Xi^{-}_{\epsilon}(q_1)
=\Xi^{-}_{\epsilon}(q_2)\\
&=\Xi^{+}_{\epsilon}(q_2)
=\Xi^{+}_{\epsilon}(q_3)=\Xi^{-}_{\epsilon}(q_3)
=\Xi^{-}_{\epsilon}(q_4)=\Xi^{+}_{\epsilon}(q_4)=\Xi_{\epsilon}(q_4).
    \end{aligned}
\end{equation}
This implies that the cross ratio $([\zeta_1],[\eta_1];[\zeta_2],[\eta_2]) \in {\rm Per}(E_0)$ which concludes the proof the claim.

\end{proof}

\section{Asymptotics of orbit growth}
\subsection{Notations}\label{sect:notations}
For the reader's convenience, we introduce notations following \cite{Link} and \cite{Roblin} that will be used in the sequel. 

Let $QD^1(S)$ be space of unit area quadratic differentials, which can be identified with the cotangent bundle of $\T(S)$ and $\MF(S)$ be the space of measured foliations on $S$. We denote $\mathcal{UE}$ the set of unique ergodic points in $\PGM$. The natural projections involved here are listed by
\begin{equation}
    \begin{aligned}
        \pi: QD^1(S) \to \T(S),\\
        \mathcal{V}: QD^1(S) \to \MF(S), q \mapsto \mathcal{V}(q),\\
        \mathcal{H}: QD^1(S) \to \MF(S), q \mapsto \mathcal{H}(q),\\
        [ \cdot ]: \MF(S) \to \PMF(S), \xi \mapsto [\xi].
        \end{aligned}
\end{equation}
We will use $\xi,\eta,\dots$ to denote boundary points in $\PGM$. When $\xi \in \PMF(S)$, $\xi$ is sometimes also referred to one of its preimages under $[\cdot]$ if there are no confusion. Define $\mathfrak{F}$ to be  $(\PGM)^2 \setminus \Delta$ where $$\Delta = \{([\lambda],[\eta]) \in (\PGM)^2: \exists \gamma \in \MF(S)-\{0\}, i(\lambda,\gamma)=i(\eta,\gamma)=0\},$$ then every pair $(\zeta, \psi) \in \mathfrak{F}$ determines a unique unparameterized Teichm\"{u}ller geodesic $\mathcal{L}(\xi,\psi)$, called {\it the Teichm\"{u}ller geodesic determined by the pair $(\zeta,\psi)$}, by taking the optimal geodesic line for $(\zeta,\psi)$.
When oriented naturally, it is also called {\it the oriented geodesic from $\zeta$ to $\psi$}. 
When $\zeta$ and $\psi$ belong to $\mathcal{UE}$, $\mathcal{L}(\zeta,\psi)$ is exactly 
$$\mathcal{L}(\zeta,\psi) = \{\pi(q): q \in QD^1(S), [\mathcal{H}(q)] = [\zeta], [\mathcal{V}(q)] = [\psi]\}.$$ 
 For any point $x \in \T(S)$, the unique unit holomorphic quadratic differential $q$ on $x$ with $\Psi(q) \in (\xi,\eta) \in \mathcal{UE} \times \mathcal{UE}$ will be denoted by $q(x;\xi,\eta)$. 

For any subset $A \subset \T(S)$ and $r > 0$, {\it the $r-$neighborhood $\mathcal{N}_r(A)$} of $A$ is the subset consisting of points in $\T(S)$ which have distance at most $r$ to some point in $A$. When $A$ is a single point, we will use $B(x,r)$ to denote the open $r-$ball centered at $x$. 

\noindent{\bf Shadows.}~~~We define the {\it (measurable) Shadow} of a ball $B(y,r)$ from $x$ to $\PGM$ for $r >0$ to be $$\mathcal{O}_r(x,y) = \{\beta \in \mathcal{UE}: \mathcal{L}(x,\beta) \cap B(y,r) \neq \emptyset\} \subset \mathcal{UE},$$ where $\mathcal{L}(x,\beta)$ is the unique optimal geodesic for $\beta$ passing through $x$. Since $\beta \in \mathcal{UE}$, $\mathcal{L}(x,\beta)$ is also the geodesic ray starting from $x$ so that the projectivized vertical measured foliation of the determined quadratic differential is $\beta$. One can definitely define the shadow to be all points in the boundary satisfying the same condition, but in order to unify notation, we choose the above definition. Let $\zeta \in \PGM$ and $y \in \T(S)$, define 
$$ \mathcal{O}_{r}(\zeta,y) = \{\eta \in \mathcal{UE}: \mathcal{L}(\zeta,\eta) \cap B(y,r) \neq \emptyset\}.$$ 
For $r >0$ and $x,y \in \T(S)$ with $d(x,y) > 2r$, define 
\begin{equation}
    \begin{aligned}
        &\mathcal{O}_r^{+}(x,y) = \{\beta \in \mathcal{UE}: \exists z \in B(x,r), such~that~ \mathcal{L}(z,\beta) \cap B(y,r) \neq \emptyset\};\\
        &\mathcal{O}_r^{-}(x,y) = \{\beta \in \mathcal{UE}: \forall z \in B(x,r), \mathcal{L}(z,\beta) \cap B(y,r) \neq \emptyset\}.        
        \end{aligned}
\end{equation}
The following lemma should be compared with Ballmann's Lemma in \cite[Lemma 2.1]{Link}.
\begin{lemma}\label{lemma:anaBallmann}
    Let $\zeta,\kappa \in \mathcal{UE}$ be two projective measured foliations determining a Teichm\"{u}ller geodesic and $o \in \T(S)$ be any point. Let $r> 0$, $\epsilon > 0$ and the Teichm\"{u}ller geodesic line determined by $\zeta,\kappa$ be $\mathcal{L}(\zeta,\kappa)$. Suppose that $d(o,\mathcal{L}(\zeta,\kappa)) = r$. Then there exist an open neighborhood $O_1 \times O_2 \subset (\mathcal{UE})^2$ of $(\zeta,\kappa)$ such that $$\forall (\mathcal{F}_1,\mathcal{F}_2) \in O_1 \times O_2, \|d(o,\mathcal{L}(\mathcal{F}_1,\mathcal{F}_2))-r\| \leq \epsilon,$$ and every pair $(a,b) \in O_1 \times O_2$ determines a Teichm\"{u}ller geodesic.  
\end{lemma}
\begin{proof}
    By \cite[Lemma 1.4.3 ]{KaiMasur}, the distance function on $\PMF(S) \times \PMF(S) - \Delta$ is continuous, so the first part of the lemma follows. Here $\PMF(S)$ is equipped with its natural topology instead of induced topology from $\PGM$. For the second part, if not, then there are measured foliations $\xi_n \to \zeta$, $\eta_n \to \kappa$ and simple closed curves $\gamma_n \in \mathcal{C}$ such that $i(\xi_n,\gamma_n)= i(\eta_n,\gamma_n)=0$. We can assume $\gamma_n \to \iota$ with $\iota \in \MF(S)$, then $i(\iota,\zeta) = i(\iota,\kappa) = 0$. On the other hand, $\iota$ cannot be minimal foliation (that is, a foliation without simple closed leaves) otherwise $\zeta$ is topologically equivalent to $\kappa$ (\cite{Rees}) and has zero intersection with it, which contradicts the assumption. So there is a simple closed curve $\gamma \in \mathcal{C}$ such that  $i(\gamma,\zeta) = i(\gamma,\kappa) = 0$ which also contradicts the assumption.
\end{proof}
\begin{corollary}\label{coro:openness}
    Let $\mathcal{O}_r(\zeta,y)$ be defined above. Then  $\mathcal{O}_r(\zeta,y)$ is open in $\mathcal{UE}$, hence it is $\nu_x$-measurable for all $x \in \T(S)$.
    \end{corollary}
\begin{proof}
    Since $\nu_x$ is a complete measure by assumption, $\mathcal{UE}$ is measurable. So the result follows from Lemma \ref{lemma:anaBallmann} if $\zeta \in \mathcal{UE}$. Otherwise, one shall use Lemma \ref{lemma:convergence}. 
\end{proof}

\begin{remark}\label{rem:discontinuity}
    \begin{itemize}
        \item[(1)] The set $\mathcal{O}_{r}(x,y)$ is open in $\mathcal{UE}$, hence $\nu_x$-meaurable. 
        \item[(2)] In general, it might not be true that, if $x_n$ converges to $\xi$ in $\GM$, then $\lim_{n \to \infty} \mathcal{O}^{\pm}_{r}(x_n,y) = \mathcal{O}_r(\xi,y)$. 
        \item[(3)] When $\xi \in \PGM$, we define $\mathcal{O}^{\pm}_r(\xi,y) = \mathcal{O}_r(\xi,y)$.
    \end{itemize}
\end{remark}
We will also use the following modified boundary of $\mathcal{O}_r(\xi,x)$ as in \cite{Link},
$$\tilde{\partial} \mathcal{O}_r(\xi,x) = \{\eta \in \mathcal{UE}-\{\xi\}: d(x,\mathcal{L}(\xi,\eta)) = r\}.$$
Notice that, by Corollary \ref{coro:openness}, $\tilde{\partial} \mathcal{O}_r(\xi,x)$ is also $\nu_x$-measurable.

\noindent{\bf Cones.}~~~Let $x \in \T(S)$ and $U \subset \PGM$. We denote by $\SE_x(U)$ the set of all points $y$ in $\T(S)$ such that there is $\xi \in U$ such that $y \in \mathcal{L}(x,\xi)$. 

For $r > 0$, following \cite{Link}, define \textit{the cones}  $$C_r^{+}(x,U) = \mathcal{N}_r(\cup_{z \in B(x,r)}\SE_z(U)) = \{y \in \T(S): \mathcal{O}_r^{+}(x,y) \cap U \neq \emptyset\},$$ and  $$C_r^{-}(x,U) = \{y \in \T(S): \mathcal{O}_r^{+}(x,y) \subset U\}.$$

\noindent{\bf Certain useful Borel sets.} The following sets are useful for us. For $r > 0$ and $x,y \in \T(S)$ with $d(x,y) > 6r$, the {\it corridor} $L_r(x,y)$ is, intuitively, the subset of $\PGM \times \PGM$  consisting of $(\xi,\eta) \in \mathcal{UE} \times \mathcal{UE}$ with $\xi \neq \eta$ so that the unique (oriented) Teichm\"{u}ller geodesic $\mathcal{L}(\xi,\eta)$ determined by the ordered pair $(\xi,\eta)$ passes first through $B(x,r)$ and then $B(y,r)$. 

To define $L_r(x,y)$ rigorously, we use the nearest point projection. Let $x \in \T(S)$ and $\ell=\mathcal{L}(\xi,\eta)$ be a Teichm\"{u}ller geodesic line. Denote $$\Pi_{\ell}(x) = \{z \in \ell: d(x,z) = d(x,\ell)\},$$ then $\Pi$ is called {\it the nearest point projection}. Notice that, in \cite{Roblin} and \cite{Link}, $\Pi_{\ell}(x)$ is always a single point because the distance function on ${\rm CAT(0)}$-space is convex. However, this is no longer true since the Teichm\"{u}ller distance is in general not convex (see \cite{BourqueRafi}). Therefore $\Pi_{\ell}(x)$ is a nonempty closed subset which may contain more than one point. Nevertheless, the diameter of $\Pi(x)$ is bounded by $2d(x,\ell)$. Recall that each ordered pair $(\xi,\eta) \in \mathcal{UE} \times \mathcal{UE}$ defines not only a unique unparameterized geodesic line  $\mathcal{L}(\xi,\eta)$, but also an orientation $\overrightarrow{\xi\eta}$ on it pointing to $\eta$. So for any two distinct points $p,q$ in $\mathcal{L}(\xi,\eta)$, we say $p$ is on the left of $q$ if the oriented geodesic line $\overrightarrow{pq}$ is positive with respect to the orientation. For $x \in \T(S)$ and an oriented geodesic $\mathcal{L}(\xi,\eta)$, we define \textbf{the leftmost nearest projection}
$w(x;\xi,\eta)$ to be the unique unit holomorphic quadratic differential on $z$ so that $[\mathcal{H}(w(x;\xi,\eta))] = \xi, [\mathcal{V}(w(x;\xi,\eta))] = \eta$, $z = \pi(w(x;\xi,\eta)) \in \Pi_{\mathcal{L}(\xi,\eta)}(x)$ and $z$ is leftmost in $\Pi_{\mathcal{L}(\xi,\eta)}(x)$. Note that here we use the fact that $\Pi_{\mathcal{L}(\xi,\eta)}(x)$ is a closed set so that the leftmost make sense.

The \textit{corridor} $L_r(x,y)$ for $d(x,y) > 6r$ is then defined to be the subset of $\mathcal{UE} \times \mathcal{UE}$ consisting of $(\xi,\eta)$ with $d(x,\mathcal{L}(\xi,\eta)) < r, d(y,\mathcal{L}(\xi,\eta)) < r$ and so that $\Pi_{\mathcal{L}(\xi,\eta)}(x)$ is on the left of $\Pi_{\mathcal{L}(\xi,\eta)}(y)$ with $md(\Pi_{\mathcal{L}(\xi,\eta)}(x), \Pi_{\mathcal{L}(\xi,\eta)}(y)) > r$, where $md(A,B)$ stands for the minimal distance (rather than the Hausdorff distance) between two sets $A,B$. By Lemma \ref{lemma:anaBallmann}, $L_r(x,y)$ is an open subset of $\mathcal{UE} \times \mathcal{UE}$, hence $\nu_x \otimes \nu_x$-measurable.


Let $g_t$ be the Teichm\"{u}ller geodesic flow on $QD^1(S)$. Let $x \in \T(S), A \subset \PGM, B \subset \PGM$ and $\widehat{\mathfrak{F}} = \mathcal{UE} \times \mathcal{UE}$, define
\begin{equation}\label{equation:K}
    \begin{aligned}
       & K_r(x) = \{g_sw(z;\xi,\eta): (\xi,\eta) \in \widehat{\mathfrak{F}}, d(x,\mathcal{L}(\xi,\eta)) \leq r, s \in (-\frac{r}{2},\frac{r}{2})\},\\
       & K_A^{+} := K^{+}_r(x,A) = \{g_tw(z;\xi,\eta)\in K_r(x): \eta \in A \},\\
       & K_{B}^{-} := K^{-}_r(x,B) = \{g_tw(z;\xi,\eta)\in K_r(x): \xi \in B \}.
        \end{aligned}
\end{equation}
The following is quite straightforward. 
\begin{lemma}
    If $A$ is $\nu_x$-measurable and $B$ is $\nu_y$-measurable, the sets $K_r(x)$, $K_A^{+}, K^{-}_{B}$ defined above are $\tilde{\mu}_G$- measurable.
\end{lemma}
\begin{proof}
    This follows from the continuity of the distance function at boundary points in $\mathcal{UE}$ (cf. \cite[Lemma 1.4.3]{KaiMasur}) and our choice of the leftmost nearest projection. 
\end{proof}
\begin{remark}
    The reader shall not bother too much with the measurability of these sets. In fact, we will not use the fact that $w(x;\xi,\eta)$ is at the leftmost position of $\Pi_{\mathcal{L}(\xi,\eta)}(x)$. Since $\Pi_{\mathcal{L}(\xi,\eta)}(x)$ has bounded distance depends on $d(x,\mathcal{L}(\xi,\eta))$, the reader can check that, whenever we use $w(x;\xi,\eta)$, the only fact relevant to us is $w(x,\xi,\eta) \in \Pi_{\mathcal{L}(\xi,\eta)}(x)$. So one can define $w(x;\xi,\eta)$ in another way. Namely, we define $$K'_r(x) = \{g_sq(z;\xi,\eta): (\xi,\eta) \in \widehat{\mathfrak{F}}, d(x,\mathcal{L}(\xi,\eta)) \leq r, s \in (-\frac{r}{2},\frac{r}{2}), z \in \Pi_{\mathcal{L}(\xi,\eta)}(x)\},$$
    then this set is easy to see to be measurable. Since for any $(\xi,\eta) \in \mathcal{UE} \times \mathcal{UE}$, $K'_{r}(x) \cap \mathcal{L}(\xi,\eta)$ is a compact set and the projection map $\phi: K'_{r}(x) \to \mathcal{UE} \times \mathcal{UE}$ is continuous, we can find a measurable section of $\phi$. Now we then can take $w(x;\xi,\eta)$ to be the value of the measurable section of $\phi$ at $(\xi,\eta)$ and define $K_r(x)$ as before which is definitely measurable.  
\end{remark}
\begin{remark}
    As pointed out in \cite{BourqueRafi}, the work of J. H. C. Whitehead implies that, for each $x \in \T(S)$, one can choose $r_x$, depending on $x$, such that for each geodesic $\mathcal{L}(\xi,\eta)$ passing through $B(x,r_x)$, $\Pi_{\mathcal{L}(\xi,\eta)}(x)$ is in fact a connected segment or a point. However, this is too much for us since we will only use boundedness of $\Pi_{\mathcal{L}(\xi,\eta)}(x)$ in the sequel.
    \end{remark}
\subsection{Properties of shadows and cones}
We now list all important properties which can be found in \cite{Link}. They will be used in the proof of the main theorem. Since all proofs can be carried out mutatis mutatum in our situation, we defer some of them to the appendix.

Concerning $\mathcal{O}_r(\xi,x)$ and $\mathcal{O}^{\pm}_r(x,y)$ in Section \ref{sect:notations}, we have
\begin{lemma}\label{lemma:nullboundary}
    Let $\xi \in \PGM$ and $o \in \T(S)$. Then the following set is at most countable $$\{ r > 0: \nu_o(\tilde{\partial} \mathcal{O}_r(\xi,o)) > 0\}$$
\end{lemma}
\begin{proof}\marginpar{\textcolor{blue}{Note A.1}}
    The same as the proof of \cite[Lemma 6.2]{Link}.
\end{proof}
\begin{lemma}\label{lemma:measureclosure}
    Let $\xi \in \PGM, o \in \T(S), r> 0$. Suppose that $y_n \in \GM$ converges to $\xi$ in $\PGM$. Then we have 
    \begin{itemize}
        \item[(a)] $$\limsup_{n \to \infty}(\mathcal{O}^{\pm}_r(y_n,o)) \subset (\mathcal{O}_r(\xi,o) \cup \tilde{\partial}\mathcal{O}_r(\xi,o)),$$
            \item[(b)] $$\mathcal{O}_r(\xi,o) \subset \liminf_{n \to \infty}(\mathcal{O}^{\pm}_r(y_n,o)).$$       \end{itemize}
\end{lemma}
\begin{proof}
    We first prove (b). Choose $\eta \in \mathcal{O}_r(\xi,o) \subset \mathcal{UE}$. Consider the unique geodesic $\mathcal{L}(\xi,\eta)$ determined by $\xi,\eta$ and the leftmost nearest projection $w(o;\xi,\eta)$, then $d(o,w(o,\xi,\eta)) < r$. Let $\epsilon > 0$ so that $B(w(o;\xi,\eta),\epsilon) \subset B(o,r)$. By Lemma \ref{lemma:convergence}, there is an open neighborhood $U$ of $\xi$ in $\GM$ such that, for any points $z \in U$, there is a geodesic $\mathcal{L}(z, \eta)$ with $$d(w(o;\xi,\eta), \mathcal{L}(z,\eta)) < \epsilon.$$ Notice that $z$ can be either in $\T(S)$ or in $\PGM$. Now choose $N$ such that whenever $n \geq N$, $B(y_n,r) \in U$ when $y_n \in \T(S)$ and $y_n \in U$ when $y_n \in \PGM$. Then $\eta \in \mathcal{O}_r^{-}(y_n,o)$ when $n \geq N$.
    
  We now prove part (a), that is, $$\limsup_{n \to \infty}(\mathcal{O}^{+}_r(y_n,o)) \subset \mathcal{O}_r(\xi,o) \cup \tilde{\partial}\mathcal{O}_r(\xi,o).$$ Let $\iota \in \limsup_{n \to \infty}(\mathcal{O}^{+}_r(y_n,o))$, then there is a infinite sequence, still denoted by $\{y_n\}$, so that $\iota \in \mathcal{O}^{+}_r(y_n,o)$. Denote the geodesic determined by $\xi$ and $\iota$ by $\ell= \mathcal{L}(\xi,\iota)$.  We can assume that $y_n \in \T(S)$ since if $y_n \in \PGM$, then we can take a sequence $y'_n \in \T(S)$ converges to $y$. Now consider oriented geodesics $\ell_n$ constructed in the following way. For $y_n \in \T(S)$, define $\ell_n$ to be the geodesic extending the oriented geodesic ray $-\mathcal{L}(y_n,\iota)$. Notice that $d(o,\ell_n) < r$.
  Furthermore, reparameterize $\ell_n$ in a fashion that $\beta_{\iota}(o,\ell(o))= \beta_{\iota}(o,\ell_n(0)) = 0$. Let $s_n \in \mathds{R}$ such that $\ell_n(s_n)$ is the leftmost point in $\Pi_{\ell_n}(o)$. Computations as in \cite[(26)]{Link} shows that $s_n \in  (-r,r)$. On the one hand, $y_n \to \xi$. On the other hand, passing to a subsequence, one can assume that $s_n \to s$. By Lemma \ref{lemma:convergence}, as $n \to \infty$, $d(\ell_n(s_n),\ell(s_n)) \to 0$. Therefore
  \begin{equation}
      \begin{aligned}
&d(o,\ell(\mathbb{R}) \leq  d(o,\ell(s)) \\
&\leq \lim_{n \to \infty}(d(o,\ell_n(s_n))+d(\ell_n(s_n), \ell(s_n)) + d(\ell(s_n),\ell(s)))\\
&\leq r.
\end{aligned}
  \end{equation}
    \end{proof}

\begin{lemma}\label{lemma:approx}
    Let $\xi \in \PGM$, $o \in \T(S)$ and $r > 0$ so that $\nu_o(\mathcal{O}_r(\xi,o)) > 0$ and $\nu_o(\tilde{\partial} \mathcal{O}_r(\xi,o)) = 0$. Then for every $\epsilon > 0$, there exists a open neighborhood $U_{\xi}$ of $[\xi]$ in $\GM$ such that $$\forall a \in U_{\xi},~ e^{-\epsilon}\nu_o(\mathcal{O}_r(\xi,o)) \leq \nu_{o}(\mathcal{O}_r^{\pm}(a,o)) \leq e^{\epsilon}\nu_o(\mathcal{O}_r(\xi,o)).$$ 
    \end{lemma}
  \begin{proof}\marginpar{\textcolor{blue}{Note A.2}}
  With Lemma \ref{lemma:measureclosure} at hand, the proof is the same as \cite[Corollary 6.4]{Link} by recalling from Theorem \ref{Yang2022} that $\nu_o(\mathcal{UE}) = 1$.
  \end{proof}

Concerning  $C^{\pm}_r(x,U)$ in Section \ref{sect:notations}, we can check using triangle inequalities that 
\begin{lemma}\label{lemma:CR}
    \begin{itemize}
        \item[(1)] $C^{-}_r(x,U) \subset C^{+}_r(x,U)$ and \\ $C^{-}_r(x,U) = \GM \setminus C^{+}_r(x, \PGM \setminus U)$.
        \item[(2)] For any $\rho > 0$ and $x_i,y_i \in \T(S) (i = 1,2)$ with $d(x_i,y_i) < \rho$, $ \mathcal{O}^{+}_r(x_1,y_1) \subset \mathcal{O}^{+}_{r+\rho}(x_2,y_2)$.
        \item[(3)] For any $\rho > 0$ and $x_i,y_i \in \T(S) (i = 1,2)$ with $d(x_i,y_i) < \rho$, $$y_1 \in C^{+}_r(x_1,U) \Rightarrow y_2 \in C_{r+\rho}^{+}(x_2,U).$$
        \item[(4)] For any $\rho > 0$ and $x_i,y_i \in \T(S) (i = 1,2)$ with $d(x_i,y_i) < \rho$, $$y_1 \in C^{-}_{r + \rho}(x_1,U) \Rightarrow y_2 \in C_{r}^{-}(x_2,U).$$
        \item[(5)] If $r_1 \leq r_2$, then $$ C^{+}_{r_1}(x,U) \subset C^{+}_{r_2}(x,U), C^{-}_{r_2}(x,U) \subset C_{r_1}^{-}(x,U).$$
        
        \end{itemize}
\end{lemma}
\begin{proof}
     See \cite[Lemma 6.5]{Link}.
\end{proof}

\begin{lemma}\label{lemma:LC}
Let $r > 0$, $\gamma \in G$, $x,y \in \T(S)$ and $A,B \subset \PGM$ are Borel subsets.
    \begin{itemize}
        \item[(1)] If $(\zeta,\eta) \in L_r(x,\gamma y) \cap (\gamma B \times A)$, then $$(\gamma y, \gamma^{-1}x) \in C^{+}_r(x,A) \times C^{+}_r(y,B),$$ $$ (\zeta,\eta) \in \mathcal{O}^{+}_r(\gamma y,x) \times \mathcal{O}_{r}^{+}(x,\gamma y).$$
        \item[(2)] If $(\gamma y, \gamma^{-1}x) \in C^{-}_{r}(x,A) \times C^{-}_r(y,B),$ then 
        \begin{equation}
            \begin{aligned}
                \{(\zeta,\eta) \in (\mathcal{UE})^2: \eta \in \mathcal{O}^{-}_r(x, \gamma y), \zeta \in \mathcal{O}_r(\eta, x)\}\\
                \subset L_r(x,\gamma y) \cap (\gamma B \times A).
            \end{aligned}
        \end{equation}
        
    \end{itemize}
\end{lemma}
\begin{proof}\marginpar{\textcolor{blue}{Note A.3}}
    For the part (1) and (2), the proofs are the same as proofs of \cite[Lemma 6.7, Lemma 6.8]{Link} since proofs there use only definitions of $C^{\pm}_r, \mathcal{O}^{\pm}_r$. 
\end{proof}

Concerning $K^{\pm}$, we collect the following properties. Let $$\Psi': QD^1(S) \to \PMF(S) \times \PMF(S) \atop q \mapsto ([\mathcal{H}(q)],[\mathcal{V}(q)])$$ be the map which is $\Psi$ in (Eq. \ref{equ:coordinate}) composed with the natural projection $\MF(S) \to \PMF(S)$ in the first factor.
\begin{lemma}\label{lemma:boundaryestimates}
     For all $\gamma \in G$ so that $d(x, \gamma y) \geq 12r$, we have 
        $$ \Psi'(\bigcup_{t > 0}\{K_A^{+} \cap g_{-t}\gamma K_B^{-}\}) = L_r(x, \gamma y) \cap (\gamma B \times A).$$
\end{lemma}
\begin{proof}\marginpar{\textcolor{blue}{Note A.4}}
    The same as in \cite[Lemma 6.9]{Link} simply by noticing that $d(x,\gamma y) \geq 12r$ ensures that $md(\Pi_{\mathcal{L}(\xi,\eta)}(x),\Pi_{\mathcal{L}(\xi,\eta)}(\gamma y)) > r$ as required in the definition of $L_r(x,\gamma y)$.
\end{proof}
We also need some estimates about the Bowen-Margulis measure of  $K^{\pm}$. First recall that (see (\ref{equ:BowenMargulis})), for $\xi \neq \eta \in \mathcal{UE}, x\in \T(S), u \in \mathcal{L}(\xi,\eta)$, $$\rho_x(\xi,\eta) = \beta_{\xi}(x,u) + \beta_{\eta}(x,u),$$ and (see Eq. \ref{equ:busemann}) $$ \beta_{\xi}(x,y) = \frac{1}{2}\ln \frac{\E_x(\xi)}{\E_y(\xi)}.$$ By Kerckhoff \cite{Kerc}, $$\beta_{\xi}(x,y) \leq d(x,y).$$ Hence 
\begin{equation}\label{equ:gromovproduct}
    \begin{aligned}
         \eta \in \mathcal{O}_{r}(\xi,x) \Rightarrow \rho_x(\xi,\eta) \leq 2r,\\
        (\xi,\eta) \in L_r(x,\gamma y) \Rightarrow \rho_x(\xi,\eta) \leq 2r.
    \end{aligned}
\end{equation}
By the definition of $\tilde{\mu}_G$ and the fact that $\nu_x$ fully supported on $\mathcal{UE}$, one has 

\begin{equation}\label{equ:estimatesofmeasure}
\small
    \begin{aligned}
    \tilde{\mu}_G(K_A^{+})=
\int_A(\int_{\mathcal{O}_r(\xi,x)}e^{\delta_G     \rho_x(\xi,\eta)}d\nu_x(\eta))d\nu_x(\xi)\int \mathds{1}_{K_A^{+}}(g_{t}w(x;\xi,\eta))dt\\
    = r\int_A(\int_{\mathcal{O}_r(\xi,x)}e^{\delta_G     \rho_x(\xi,\eta)}d\nu_x(\eta))d\nu_x(\xi);\\    
     \tilde{\mu}_G(K_B^{-})=
     \int_B(\int_{\mathcal{O}_r(\eta,y)}e^{\delta_G     \rho_x(\xi,\eta)}d\nu_y(\eta))d\nu_y(\xi)\int \mathds{1}_{K_B^{-}}(g_{t}w(x;\xi,\eta))dt \\    = r\int_B(\int_{\mathcal{O}_r(\eta,y)}e^{\delta_G \rho_x(\xi,\eta)}d\nu_y(\xi))d\nu_y(\eta).    \end{aligned}
\end{equation}
Combine (\ref{equ:estimatesofmeasure}) and (\ref{equ:gromovproduct}), one has 
\begin{equation}\label{equ:measureofK}
    \begin{aligned}
        r\int_A\nu_x(\mathcal{O}_r(\xi,x))d\nu_x(\xi) \leq \tilde{\nu}_G(K_A^{+}) \leq re^{2r\delta_G}\int_A\nu_x(\mathcal{O}_r(\xi,x))d\nu_x(\xi),\\
         r\int_B\nu_y(\mathcal{O}_r(\eta,y))d\nu_y(\eta) \leq \tilde{\nu}_G(K_B^{-}) \leq re^{2r\delta_G}\int_B\nu_y(\mathcal{O}_r(\eta,y))d\nu_y(\eta).    
         \end{aligned}
\end{equation}
Recall that for $x \in \T(S)$ and the geodesic line $\mathcal{L}(\xi,\eta)$, we have the nearest point projection $\Pi_{\mathcal{L}(\xi,\eta)}(x)$ and the leftmost point $w(x;\xi,\eta)$.
\begin{lemma}\label{lem:integral}
Let $T_0 > 12r$, $\gamma \in G$ with $d(x, \gamma y)  \geq 12r$, $T > T_0 + 6r$, $s \in (-\frac{r}{2},\frac{r}{2})$ and $(\xi,\eta) \in L_r(x,\gamma y)$.
\begin{itemize}
    \item[(1)] If $d(x,\gamma y) \in (T_0+3r,T]$, then\begin{equation}
        \begin{aligned}
            \int^{T+3r}_{T_0}e^{\delta_G t}\mathds{1}_{K_r(\gamma y)}(g_{t+s}w(x;\xi,\eta))dt \geq re^{-3r\delta_G}e^{\delta_Gd(x,\gamma y)}.
        \end{aligned}
    \end{equation} 
    \item[(2)] \begin{equation}
        \begin{aligned}
            \int^{T-6r}_{T_0}e^{\delta_G t}\mathds{1}_{K_r(\gamma y)}(g_{t+s}w(x;\xi,\eta))dt \leq re^{3r\delta_G}e^{\delta_Gd(x,\gamma y)}.
        \end{aligned}
    \end{equation}
    \item[(3)] If $d(x,\gamma y) \leq T_0-3r$ or $d(x,\gamma y) > T$, then 
    \begin{equation}
    \begin{aligned}
            \int^{T-6r}_{T_0}e^{\delta_G t}\mathds{1}_{K_r(\gamma y)}(g_{t+s}w(x;\xi,\eta))dt = 0.
        \end{aligned}
    \end{equation}
    \end{itemize}
    
\end{lemma}
\begin{proof}
     Consider $w(x;\xi,\eta)$ and $w(\gamma y; \xi,\eta)$. There is $\tau > 0$ such that $w(\gamma y;\xi,\eta) = g_{\tau}w(x;\xi,\eta)$. And since $(\xi,\eta) \in L_r(x,\gamma y)$, $\|d(x,\gamma y) - \tau \| < 2r$. Note that 
    $$K_r(\gamma y; \xi,\eta) = \{g_sw(\gamma y;\xi,\eta): s \in (-\frac{r}{2},\frac{r}{2}) \},$$ and $g_{t+s}w(x;\xi,\eta) \in K_r(\gamma y; \xi,\eta)$ if and only if $t+s-\tau \in (-\frac{r}{2},\frac{r}{2})$. Hence if $\tau-s-\frac{r}{2} \geq T_0$ and $ \tau -s + \frac{r}{2} \leq T + 3r$, then 
\begin{equation}
    \begin{aligned}
&\int^{T+3r}_{T_0}e^{\delta_G t}\mathds{1}_{K_r(\gamma y;\xi,\eta)}(g_{t+s}w(x;\xi,\eta))dt\\
&\geq \int^{\tau -s + \frac{r}{2}}_{\tau-s-\frac{r}{2}} e^{\delta_G t}dt \\
&\geq re^{-3r\delta_G}e^{\delta_Gd(x,\gamma y)}.    \end{aligned}
\end{equation}
On the other hand, it is easy to see that $d(x,\gamma y) \in (T_0+3r,T]$ implies that $\tau-s-\frac{r}{2} \geq T_0$ and $ \tau -s + \frac{r}{2} \leq T + 3r$, so we have $(1)$.
For $(2)$, if $g_{t+s}w(x;\xi,\eta) \in K_r(\gamma y)$, then $$t+s-\tau \leq \frac{r}{2}, \tau - (t+s) \leq \frac{r}{2}.$$ So 
\begin{equation}
    \begin{aligned}
        &\int^{T-6r}_{T_0}e^{\delta_G t}\mathds{1}_{K_r(\gamma y)}(g_{t+s}w(x;\xi,\eta))dt \\
        &\leq \int^{\tau -s + \frac{r}{2}}_{\tau -s -\frac{r}{2}}e^{\delta_G t}dt\\
        &\leq re^{\delta_G(\tau -s + \frac{r}{2})}\\
        &\leq re^{3r\delta_G}e^{\delta_Gd(x,\gamma y)}.
        \end{aligned}
\end{equation}
Finally for $(3)$, notice that if $d(x,\gamma y) \leq T_0-3r$, then $\tau -s+\frac{r}{2} \leq T_0$ and if $d(x,\gamma y) \geq T$, then $\tau -s-3r \geq T-6r$.
\end{proof}
We then have the following estimations from Lemma \ref{lemma:boundaryestimates}.
\begin{lemma}\label{lem:measureofKPM}
Let $\gamma \in G$ and $d(x,\gamma y) \geq 12r$. Let $r > 0$, then 
    \begin{equation}
        \begin{aligned}
            &\int_{L_{r}(x,\gamma y) \cap (\gamma B \times A)}e^{\delta_G \rho_x(\xi,\eta)}d\nu_x(\xi)d\nu_y(\eta) \int^{\frac{r}{2}}_{-\frac{r}{2}}\mathds{1}_{K_r(\gamma y)}(g_{t+s}w(x;\xi,\eta))ds  \\
            &= \tilde{\mu}_{G}(K_A^{+} \cap g_{-t}\gamma K_B^{-}) .        \end{aligned}
        \end{equation}
\end{lemma}

\subsection{Subexponential growth of orbits in cones}
In this section, we prove a counting lemma which will be a substitute of \cite[Lemma 6.6]{Link}. The main reason that we can only hope a counting result like this is that $\T(S)$ is not $CAT(0)$, so it seems that unlike \cite[Lemma 6.6]{Link}, we cannot have a finiteness counting lemma.
\begin{proposition}\label{lemma:subexponential}
   Let $G$ be a non-elementary subgroup of divergence type of $\M(S)$ with the critical exponent $\delta_G$. Let $x,y \in \T(S)$ and $r > 0$. Given two open subsets $\widehat{W} \subset \GM$ and $U \subset \PGM$. 
    \begin{itemize}
        \item[(1)] If $A$ be a subset of $\PGM$ whose closure $\bar{A}$ in $\PGM$ is contained in $\widehat{W} \cap \PMF(S)$, then $$\sharp\{g \in G: gy \in (C^{\pm}_r(x, A) \setminus \widehat{W}) \cap B(R,x)\} = o(e^{\delta_G R});$$
        \item[(2)] If $B \subset \GM$ and the closure $\overline{B}$ of $B$ in $\PGM$ satisfies $\overline{B} \cap \PGM \subset U$, then $$\sharp\{g \in G: gy \in (B \setminus C^{\pm}_r(x, U)) \cap B(R,x)\} = o(e^{\delta_G R}).$$
    \end{itemize}
\end{proposition}

This will follow from the following lemma.
\begin{lemma}\label{lemma:limitsofgeodesics}
Let $x\in \T(S)$ and let $z_n\in \T(S)$ be a sequence of points with $d(z_n,x)\to \infty$ lying on geodesic rays starting at points $x_n$ within distance $r$ of $x$ with uniquely ergodic vertical foliations $\zeta_n \in \mathcal{UE}$. Assume that $\zeta_n$ converge in the Gardiner-Masur compactification to  $\zeta \in \mathcal{UE}\subset \PGM$. Then also $z_n\to \eta$. 
\end{lemma}
\begin{proof}
We first claim that any Gromov-Hausdorff accumulation point of the geodesic segments $[x_n,\zeta_n)$ is $[y,\zeta)$ for some $y\in B_{r}(x)$. This follows from Klarreich \cite[Proposition 5.1]{klarreich2022boundary}, Miyachi \cite{Miyachi_II} (Corollary 1 in Section 6.1), Masur's two boundaries theorem \cite{Masur} and a diagonal argument. Now since $[x_n,z_n]\subset [x_n,\zeta_n)$, any Gromov-Hausdorff limit of $[x_n,z_n]$ is $[y,\zeta)$. 
Since $\zeta$ is the only accumulation point in $\PGM$ of $[y,\zeta)$ the result follows.
  
\end{proof}
This immediately implies the following.
\begin{corollary}\label{corollary:closureofsectors}
  For any $A\subset \PGM$ contained in the closure of uniquely ergodic points, we have $\overline{C^\pm(x,A)}\cap \mathcal{UE}= \overline{A}\cap \mathcal{UE}$   
\end{corollary}
We will also need the following.
\begin{lemma} \label{lemma:coarseequidistribution}
    Assume the action $G\curvearrowright \T(S)$ has purely exponential growth. 
    For $x\in \T(S)$  let $A_{K,R}(x)$ be the set of $y\in \T(S)$ with $R-K\leq d(x,y)<R$. Then for any $x,y\in \T(S)$, for large enough $K>0$ there is a $C>0$ such that for any Borel set $A\subset \GM$ we have
$$\lim \sup_{R\to \infty}\frac{|A\cap A_{K,R}(x) \cap Gy|}{|A_{K,R}(x) \cap Gy|}\leq C\nu(\overline{A})$$
\end{lemma}
\begin{proof}
    Let $K>0$ be large enough so that $\nu_x(\mathcal{O}_K(x,gy))\asymp_\times e^{-\delta d(x,gy)}$ \cite{yang2022conformal} and $|A_{K,R}(x) \cap Gy|\asymp_\times e^{\delta R}$. Here the notation $A(R) \asymp_\times B(R)$ is $A(R) \asymp_{x,y,K} B(R)$ which means $\lim_{R \to \infty} \frac{A(R)}{B(R)} = c \neq 0$ with $c$ depends only on the choice of $x,y$ and $K$.
    Since the number of orbit points in any ball of finite radius is finite, a point of $\PGM$ is in at most $D$  shadows $\mathcal{O}_K(x,gy), gy\in A_{K,R}(x)$.

    Thus,
$$\sum_{\{g: gy\in A_{K,R}(x)\cap A\}}\nu_x(\mathcal{O}_K(x,gy))\leq D
\mu(\cup_{\{g:gy\in A_{K,R}(x)\cap A\}}\mathcal{O}_K(x,gy))$$
Moreover, by Lemma \ref{lemma:convergence} any uniquely ergodic point of
$$\cap_{n\in \mathbb{N}z}\cup_{\{g:gy\in A_{K,R}(x)\cap A\}}\mathcal{O}_K(x,gy)$$ lies in  $\overline{A}$.

Thus, for large enough $n$ we have

$$\nu(\cup_{\{g:gy\in A_{K,R}(x)\cap A\}}\mathcal{O}_K(x,gy))\leq 2\nu_x\overline{A})$$ so

$$\frac{|A\cap A_{K,R}(x) \cap Gy|}{|A_{K,R}(x) \cap Gy|}\simeq e^{-hR}|A\cap A_{K,R}(x) \cap Gy|\simeq \sum_{\{g:gy\in A_{K,R}(x)\cap A\}}\nu_x(\mathcal{O}_K(x,gy))$$ $$\leq D
\nu_x(\cup_{\{g:gy\in A_{K,R}(x)\cap A\}}\mathcal{O}_K(x,gy))\leq 2D\nu_x(\overline{A})$$ where the implied constants depend on $x,y$ and $K$ but not on $R$.

\end{proof}
\begin{proof} [Proof of Proposition \ref{lemma:subexponential}]
If the action has purely exponential growth the result follows from Lemma \ref{lemma:coarseequidistribution} and Corollary \ref{corollary:closureofsectors}. Otherwise, \cite[Theorem B]{Yang_scc} implies that $e^{-\delta R}|B_{R}(x)\cap G y|\to 0$ and the result follows trivially.

\end{proof}

\subsection{Asymptotic formulas}
\begin{theorem}\label{main:counting}
   Let $S$ be a closed surface of genus $g$ and $G$ be a non-elementary subgroup of the mapping class group $\M(S)$. Let $x,y$ be two points in the Teichm\"{u}ller space $\T(S)$. Denote the associated Bowen-Margulis measure on $QD^1(S)/G$ by $\mu_G$ and the critical exponent of $G$ by $\delta_G$. 
   \begin{itemize}
       \item[(1)] If $G$ is convergent, then $$\lim_{R \to \infty}\frac{\sharp\{g \in G: d(x, gy) \leq R\}}{e^{\delta_G R}} = 0.$$
       \item[(2)] If $G$ is divergent and $\mu_G$ is infinite, then $$\lim_{R \to \infty}\frac{\sharp\{g \in G: d(x, gy) \leq R\}}{e^{\delta_G R}} = 0.$$ 
       \item[(3)] If $\mu_G$ is finite, then $$\lim_{R \to \infty}\frac{\sharp\{g \in G: d(x, gy) \leq R\}}{e^{\delta_G R}} =c,$$ where $c = \frac{1}{\mu_G(QD^1(S)/G)}$.        \end{itemize}
\end{theorem}
As we have all necessary modifications at hand, our proof will be adapted from the arguments in Roblin \cite{Roblin} and Link \cite{Link} directly. One observation which is crucial to us is that, the argument in \cite{Link} works well even if the number of the orbit points in a sector which are far away from the boundary is subexponential. This observation enable us to use the estimates in previous subsection to replace \cite[Lemma 6.6]{Link}. 
\begin{proof}
We only prove part $(3)$ in the next subsection. The proof of part $(2)$ is the same as part $(3)$ where the first conclusion of Corollary \ref{coro:mixing} is replaced by the second conclusion and part $(1)$ is the same as part $(3)$ using dissipativity for convergent groups.
\end{proof}
\subsubsection{Proof of (3).}
 The part $(3)$ in Theorem \ref{main:counting} follows from the following more general theorem by taking $f$ to be $\mathds{1}_{\GM \times \GM}$. For $z \in \T(S)$, we use $\mathcal{D}_x$ to denote the Dirac measure of $x$. 
\begin{theorem}\label{thm:equidistribution}
    Let $G$ be a non-elementary subgroup of $\M(S)$ with finite Bowen-Margulis measure $\mu_G$. Let $x,y \in \T(S)$. Denote the Patterson-Sullivan measures by $\nu_x, \nu_y$, respectively, and $\lVert \mu_G\lVert = \mu_G(QD^1(S)/G)$. Then we have the following weak convergence in $\\C(\GM \times \GM)^{*}$ as $R \to \infty$, 
    $$\delta_G e^{-\delta_G R} \sum_{g \in G \atop d(x, gy) \leq R}\mathcal{D}_{gy} \otimes \mathcal{D}_{g^{-1}x} \rightharpoonup \frac{1}{\lVert \mu_G \lVert}\nu_x \otimes \nu_y.$$
\end{theorem}

 Let $x,y \in \T(S)$ and set $$\mu^{R}_{x,y} = \delta_G e^{-\delta_G R} \sum_{g \in G \atop d(x, gy) \leq R}\mathcal{D}_{gy} \otimes \mathcal{D}_{g^{-1}x}.$$
 The rest of this section devotes to the proof of Theorem \ref{thm:equidistribution}. However, since \cite[Section 8]{Link} presented the proof in a very detailed  way and we will use most of the arguments, except certain necessary technical modifications, the proof presented here is not self-contained. We only pointed out modifications and the reader is referred to \cite[Section 8]{Link} for more details.\\  
{\bf Standing Assumption:} $G$ is a nonelementary subgroup of $\M(S)$ with finite Bowen-Margulis measure $\mu_G$.
Recall that, since conical points belong to $\mathcal{UE}$ and for any $\xi \in \PGM$, $(\xi,\eta)$ is a filling pair for any $\eta \in \mathcal{UE}$, so the conditions in the following proposition make sense. 
\begin{proposition}\label{prop:small}
    Let $\epsilon > 0$, $(\xi_0,\eta_0) \in (\PGM)^2$, $x, y \in \T(S)$ with the trivial stabilizer in $G$ and $x \in \mathcal{L}(\xi_0,\xi^+), y \in \mathcal{L}(\eta_0,\eta^{+})$ for conical points $\xi^{+},\eta^{+}$ in the support of $\nu_x,\nu_y$, respectively. Then there exist open neighborhoods $U,W \subset \PGM$ containing $\xi_0$ and $\eta_0$, respectively, such that for all Borel subsets $A  \subset U, B \subset W$,
    \begin{equation*}
        \begin{aligned}
            \limsup_{R \to \infty} \mu^{R}_{x,y}(C^{-}_{1}(x,A) \times C^{-}_{1}(y,B)) \leq e^{\epsilon}\nu_x(A)\nu_y(B)/\lVert\mu_G\lVert, \\
             \liminf_{R \to \infty} \mu^{R}_{x,y}(C^{+}_{1}(x,A) \times C^{+}_{1}(y,B)) \geq e^{-\epsilon}\nu_x(A)\nu_y(B)/\lVert\mu_G\lVert.        \end{aligned}
    \end{equation*}
\end{proposition}

\begin{proof}
Throughout the whole proof, the reader is refer to Section \ref{sect:notations} for all notations. 
Let $\epsilon > 0$. Let $\rho = \min\{d(x,\gamma x),d(y,\gamma y): \gamma \in G \setminus \{id\}\}$. Choose $0 < r < \min\{\frac{1}{10}, \frac{\rho}{3}, \frac{\epsilon}{60\delta_G}\}$ so that 
$$ \nu_x(\tilde{\partial} \mathcal{O}_r(\xi_0,x)) = \nu_y (\tilde{\partial} \mathcal{O}_r(\eta_0,y)) = 0.$$
The existence of such $r$ is given by Lemma \ref{lemma:nullboundary}.

\noindent \underline{Claim 1:} $$ \nu_x(\mathcal{O}_r(\xi_0,x)) \nu_y (\mathcal{O}_r(\eta_0,y)) > 0.$$

\begin{proof}[Proof of Claim 1]
Indeed, first notice that by the assumption, $\xi^{+} \in \mathcal{O}_r(\xi_0,x),\\ \eta^{+} \in \mathcal{O}_r(\eta_0,y)$ and $\xi^{+},\eta^{+}$ are in the support of measures. So we need to show that there are open subsets $x \in O_1, y \in O_2$ contained in corresponding sets. The claim is then exactly given by Corollary \ref{coro:openness}. 
\end{proof}
Thanks to Lemma \ref{lemma:approx}, there exist open neighborhoods $\widehat{U},\widehat{W}$ of $\xi_0$ and $\eta_0$, respectively, such that whenever $(a,b) \in \widehat{U} \times \widehat{W}$, one has 
\begin{equation}\label{equ:measureofO}
    \begin{aligned}
        e^{-\frac{\epsilon}{30}}\nu_x(\mathcal{O}_r(\xi_0,x)) \leq \nu_x(\mathcal{O}^{\pm}_r(a,x)) \leq e^{\frac{\epsilon}{30}}\nu_x(\mathcal{O}_r(\xi_0,x)),\\
e^{-\frac{\epsilon}{30}}\nu_y(\mathcal{O}_r(\eta_0,y)) \leq \nu_y(\mathcal{O}^{\pm}_r(b,y)) \leq e^{\frac{\epsilon}{30}}\nu_x(\mathcal{O}_r(\eta_0,y)).       
        \end{aligned}
\end{equation}
Now we consider small Borel subsets. Namely, let $U,W$ be two open neighborhoods of $\xi_0,\eta_0$ in $\PGM$ so that the closures $\bar{U} \subset \widehat{U} \cap \PGM$, $\bar{W} \subset \widehat{W} \cap \PGM$. Take Borel subsets $A \subset U$ and $B \subset W$. 

We then apply Roblin's method that consists of giving bounds for \begin{equation*}
    \begin{aligned}        
&\int^{T + 3r}_{T_0}e^{\delta_{G} t}\sum_{\gamma \in G}\tilde{\mu}_G(K_A^{+} \cap g_{-t}\gamma K_B^{-})dt,\\
&\int^{T -6r}_{T_0}e^{\delta_{G} t}\sum_{\gamma \in G}\tilde{\mu}_G(K_A^{+} \cap g_{-t}\gamma K_B^{-})dt
\end{aligned}
\end{equation*}
from two approaches, namely mixing of the geodesic flow and direct computations.

\noindent \underline{Claim 2:} Let $T_0 > 20r$, $T - 6r > T_0$ and put $M = r^2\nu_x(\mathcal{O}_r(\xi_0,x))\nu_y(\mathcal{O}_r(\xi_0,y))$. Then for $t \gg T_0$,
\begin{small}
\begin{equation}
    \begin{aligned}
    (e^{\delta_G(T-6r)}-e^{\delta_G T_0})M\nu_x(A)\nu_x(B) \leq e^{\frac{2\epsilon}{5}}\lVert \mu_G \lVert \delta_G \int^{T-6r}_{T_0}e^{\delta_G t}\sum_{\gamma \in G}\tilde{\mu}_G(K^{+} \cap g_{-t}\gamma K^{-})dt;\\
    (e^{\delta_G(T+3r)}-e^{\delta_G T_0})M\nu_x(A)\nu_x(B) \geq e^{-\frac{8\epsilon}{15}}\lVert \mu_G \lVert \delta_G \int^{T+3r}_{T_0}e^{\delta_G t}\sum_{\gamma \in G}\tilde{\mu}_G(K^{+} \cap g_{-t}\gamma K^{-})dt.    \end{aligned}
\end{equation}
\end{small}
\begin{proof}[Proof of Calim 2]
    By the definition of $K^{+} = K_A^{+}, K^{-}= K_B^{-},K_r$ and $\rho$, one has $$\forall \gamma \in G \setminus \{id\},~~~K_r(x) \cap \gamma K_r(x) = \emptyset, K_r(y) \cap \gamma K_r(y) = \emptyset.$$ So under the projection $QD^1(S) \to QD^1(S)/G$, both $K^{\pm} = K^{\pm}_r$ project injectively to $QD^1(S)/G$. One then apply Corollary \ref{coro:mixing} to have $$\lim_{t \to \infty}\sum_{\gamma \in G}\tilde{\mu}_G(K^{+}\cap g_{-t}\gamma K^{-}) = \frac{\tilde{\mu}_G(K^{+})\tilde{\mu}_G(K^{-})}{\mu_G(QD^1(S)/G)}.$$
    Now choose $T_0 > 12r$ and $t \gg T_0$, computations \marginpar{\textcolor{blue}{Note A.5}}
    in the proof of \cite[Proposition 8.2, Page 827-828]{Link} together with (\ref{equ:measureofK}), (\ref{equ:measureofO}) conclude the proof of the claim.
    
\end{proof}

\noindent \underline{Claim 3:} 
\begin{equation}
    \begin{aligned}
        \liminf_{R \to \infty} \mu^{R}_{x,y}(C^{+}_{1}(x,A) \times C^{+}_{1}(y,B)) \geq e^{-\epsilon}\nu_x(A)\nu_y(B)/\lVert\mu_G\lVert.    \end{aligned}
\end{equation}

\begin{proof}[Proof of Claim 3]
First notice that, since 
$$ \rho_{x}(\xi,\eta) = \lim_{z \to [\xi], w \to [\eta]} (d(x,z) + d(x,w) - d(z,w)),$$ we have $\rho_x(\xi,\eta) \leq 2r$ if $(\xi,\eta) \in L_r(x,\gamma y)$. Furthermore, if $T_0 > 60r$, then $K^{+} \cap g_{-t}\gamma K^{-} \neq \emptyset$ for some $t \geq T_0$ implies that $d(x,\gamma y) \geq 12r$. Meanwhile, conformality of the measure $\nu_x$ implies that $$\int_{\mathcal{O}_r^{+}(x,\gamma y)}d\nu_x(\eta)e^{\delta_G d(x,\gamma y)} \leq e^{4r\delta_G}\nu_y(\mathcal{O}_r^{+}(\gamma^{-1}x,y)), $$ here we use the fact which can be derived by the triangle inequality, that if $\eta \in \mathcal{O}_r^{+}(x,\gamma y)$, then the Busemann cocycle $\beta_{\eta}(x,y) \geq d(x,\gamma y) -4r$.
 With these remarks in mind, take $T_0 > 60r$ and since $r < 1$, we have, according to Lemma \ref{lem:measureofKPM}, Lemma \ref{lem:integral}(2), Lemma \ref{lemma:LC}, Lemma \ref{lemma:CR}(5) and the fact that $\nu_x,\nu_y$ are fully supported on $\mathcal{UE}$, 
 \begin{equation}
     \begin{aligned}
&\int^{T - 6r}_{T_0}e^{\delta_{G} t}\sum_{\gamma \in G}\tilde{\mu}_G(K^{+} \cap g_{-t}\gamma K^{-})dt =\sum_{\gamma \in G} \int^{T - 6r}_{T_0}e^{\delta_{G} t} \\
&(\int_{L_{r}(x,\gamma y) \cap (\gamma B \times A)}e^{\delta_G \rho_x(\xi,\eta)}d\nu_x(\xi)d\nu_x(\eta) \int^{\frac{r}{2}}_{-\frac{r}{2}}\mathds{1}_{K_r(\gamma y)}(g_{t+s}w(x;\xi,\eta))ds)dt\\
&\leq r^2e^{5r\delta_G }\sum_{\gamma \in G, d(x,\gamma y) \leq T }e^{\delta_Gd(x,\gamma y)}\int_{L_{r}(x,\gamma y) \cap (\gamma B \times A)}d\nu_x(\xi)d\nu_x(\eta)\\
&\leq r^2e^{5r\delta_G }\sum_{\gamma \in G, d(x,\gamma y) \leq T \atop (\gamma y, \gamma^{-1}x) \in C_1^{+}(x,A) \times C_1^{+}(y,B)}e^{\delta_Gd(x,\gamma y)}\int_{\mathcal{O}^{+}_r(\gamma y,x) \times \mathcal{O}_r^{+}(x,\gamma y)}d\nu_x(\xi)d\nu_x(\eta)\\
&\leq r^2e^{5r\delta_G }\sum_{\gamma \in G, d(x,\gamma y) \leq T \atop (\gamma y, \gamma^{-1}x) \in C_1^{+}(x,A) \times C_1^{+}(y,B)}e^{\delta_Gd(x,\gamma y)}\nu_x(\mathcal{O}^{+}_r(\gamma y,x))\nu_x(\mathcal{O}_{r}^{+}(x,\gamma y))\\
&\leq r^2e^{9r\delta_G }\sum_{\gamma \in G, d(x,\gamma y) \leq T \atop (\gamma y, \gamma^{-1}x) \in C_1^{+}(x,A) \times C_1^{+}(y,B)}\nu_x(\mathcal{O}^{+}_r(\gamma y,x))\nu_y(\mathcal{O}_{r}^{+}(\gamma^{-1}x,y)).     
\end{aligned}
 \end{equation}


Let $$\mathcal{X} = (C_{1}^{+}(x,A) \times C_{1}^{+}(y,B)) \cap (\widehat{U} \times \widehat{W})$$ and $$\mathcal{Y} =(C_{1}^{+}(x,A) \times C_{1}^{+}(y,B)) \setminus (\widehat{U} \times \widehat{W}),$$ then $$\mathcal{X} \cup \mathcal{Y} =C_{1}^{+}(x,A) \times C_{1}^{+}(y,B).$$
So 
\begin{equation}
    \begin{aligned}
&\int^{T - 6r}_{T_0}e^{\delta_{G} t}\sum_{\gamma \in G}\tilde{\mu}_G(K^{+} \cap g_{-t}\gamma K^{-})dt \\
 &\leq \underbrace{r^2e^{9r\delta_G }\sum_{\gamma \in G, d(x,\gamma y) \leq T \atop (\gamma y, \gamma^{-1}x) \in \mathcal{X}}\nu_x(\mathcal{O}^{+}_r(\gamma y,x))\nu_y(\mathcal{O}_{r}^{+}(\gamma^{-1}x,y))}_{I}\\
        &+\underbrace{r^2e^{9r\delta_G }\sum_{\gamma \in G, d(x,\gamma y) \leq T \atop (\gamma y, \gamma^{-1}x) \in \mathcal{Y}}\nu_x(\mathcal{O}^{+}_r(\gamma y,x))\nu_y(\mathcal{O}_{r}^{+}(\gamma^{-1}x,y))}_{II}\\
        &\leq I + II.
    \end{aligned}
\end{equation}
Notice that on the one hand, by (\ref{equ:measureofO}) and our choice $r < \min\{1,\frac{\epsilon}{60\delta_G}\}$, 
\begin{equation}
    \begin{aligned}
        I & \leq  r^2e^{\frac{11\epsilon}{30}}\sum_{\gamma \in G, d(x,\gamma y) \leq T \atop (\gamma y, \gamma^{-1}x) \in \mathcal{X}}\nu_x(\mathcal{O}_r(\xi_0,x))\nu_y(\mathcal{O}_{r}(\eta_0,y))\\
        &\leq e^{\frac{11\epsilon}{30}}M\frac{e^{T\delta_G}}{\delta_G}\mu^{T}_{x,y}(C_{1}^{+}(x,A) \times C_{1}^{+}(y,B)).
        \end{aligned}
\end{equation}
On the other hand, by Proposition \ref{lemma:subexponential} (1) and the fact that $\|\nu_x\| =\|\nu_y\| =1$, as $T \to \infty$ 
\begin{equation}
    \begin{aligned}
        II = o(e^{T\delta_G}).
    \end{aligned}
\end{equation}
Hence 
\begin{equation}
    \begin{aligned}
        &\int^{T - 6r}_{T_0}e^{\delta_{G} t}\sum_{\gamma \in G}\tilde{\mu}_G(K^{+} \cap g_{-t}\gamma K^{-})dt\\
        &\leq e^{\frac{11\epsilon}{30}}M\frac{e^{T\delta_G}}{\delta_G}\mu^{T}_{x,y}(C_{1}^{+}(x,A) \times C_{1}^{+}(y,B)) + e^{T\delta_G}o(1).
        \end{aligned}
\end{equation}
Put Claim 2 and these estimations together, one completes the proof of Claim 3 by dividing $Me^{\delta_G (T-6r)}\lVert\mu_G\lVert$.
\end{proof}

\noindent \underline{Claim 4:} 
\begin{equation}
    \begin{aligned}
        \limsup_{R \to \infty} \mu^{R}_{x,y}(C^{-}_{1}(x,A) \times C^{-}_{1}(y,B)) \leq e^{\epsilon}\nu_x(A)\nu_y(B)/\lVert\mu_G\lVert  \end{aligned}
\end{equation}

\begin{proof}[Proof of Claim ]
    The proof of Claim 4 is very much similar to the proof of Claim 3. We sketch the proof.

As before, we give certain remarks. Notice first that if $(\xi,\eta) \in (\mathcal{UE})^2$ then $\rho_x(\xi,\eta) \geq 0$. By Lemma \ref{lemma:LC}(2), whenever $ (\gamma y,\gamma^{-1}x) \in C_1^{-}(x,A) \times C^{-}_1(y,B) \subset C_r^{-}(x,A) \times C^{-}_r(y,B)$, $$\{(\zeta,\eta): \eta \in \mathcal{O}^{-}_r(x, \gamma y), \zeta \in \mathcal{O}_r(\eta, x)\}\\
 \subset L_r(x,\gamma y) \cap (\gamma B \times A).$$
 Furthermore, we know that since $\gamma y \in C_1^{-}(x,A) \subset C_r^{-}(x,A)$, $\mathcal{O}_r^{-}(x,\gamma y) \subset \mathcal{O}_r^{+}(x,\gamma y) \subset A \subset \widehat{U}$. As a consequence of (\ref{equ:measureofO}), if $\xi \in \mathcal{O}^{-}_r(x,\gamma y)$,
 \begin{equation}
     \begin{aligned}
&\nu_x(\mathcal{O}_r(\xi,x)) \geq e^{-\frac{\epsilon}{30}}\nu_x(\mathcal{O}_r(\xi_0,x)).
\end{aligned}
\end{equation}
If $\gamma^{-1}x \in \widehat{W}$, we further have, by the conformality of the measureS and (\ref{equ:measureofO}),
\begin{equation}
\begin{aligned}
&\int_{\mathcal{O}_r^{-}(x,\gamma y)}d\nu_x(\xi)e^{\delta_G d(x,\gamma y)} \geq \nu_y(\mathcal{O}^{-}_r(\gamma^{-1}x,y))\\
&\geq e^{-\frac{\epsilon}{30}}\nu_y(\mathcal{O}_r(\eta_0,y)).
     \end{aligned}
 \end{equation}
As in the proof of Claim 3, put the above remarks together and apply Lemma \ref{lem:measureofKPM} combine with Lemma \ref{lem:integral}(1) and the fact that $\nu_x,\nu_y$ are fully supported on $\mathcal{UE}$, we have
 \begin{equation}
     \begin{aligned}
&\int^{T+3r}_{T_0}e^{\delta_{G} t}\sum_{\gamma \in G}\tilde{\mu}_G(K^{+} \cap g_{-t}\gamma K^{-})dt \geq \sum_{\gamma \in G} \int^{T +3r}_{T_0}e^{0} \\
&(\int_{L_{r}(x,\gamma y) \cap (\gamma B \times A)}d\nu_x(\xi)d\nu_y(\eta) \int^{\frac{r}{2}}_{-\frac{r}{2}}e^{\delta_Gt}\mathds{1}_{K_r(\gamma y)}(g_{t+s}w(x;\xi,\eta))ds)dt\\
&\geq r^2e^{-3r\delta_G }\sum_{\gamma \in G, \atop T_0 + 3r \leq d(x,\gamma y) \leq T }e^{\delta_Gd(x,\gamma y)}\int_{L_{r}(x,\gamma y) \cap (\gamma B \times A)}d\nu_x(\xi)d\nu_x(\eta)\\
&\geq r^2e^{-\frac{\epsilon}{20} }\sum_{\gamma \in G,  T_0 +3r \leq d(x,\gamma y) \leq T \atop (\gamma y, \gamma^{-1}x) \in C_1^{-}(x,A) \times C_1^{-}(y,B)}e^{\delta_Gd(x,\gamma y)}\int_{\mathcal{O}_r^{-}(x,\gamma y)}d\nu_x(\xi)d\nu_x(\mathcal{O}_r(\xi,x))\\
&\geq r^2e^{-\frac{\epsilon}{12} }\sum_{\gamma \in G,  T_0 +3r \leq d(x,\gamma y) \leq T \atop (\gamma y, \gamma^{-1}x) \in C_1^{-}(x,A) \times C_1^{-}(y,B)}e^{\delta_Gd(x,\gamma y)}\int_{\mathcal{O}_r^{-}(x,\gamma y)}d\nu_x(\xi)d\nu_x(\mathcal{O}_r(\xi_0,x))\\
&\leq r^2e^{-\frac{\epsilon}{12} }\sum_{\gamma \in G, d(x,\gamma y) \leq T \atop (\gamma y, \gamma^{-1}x) \in C_1^{-}(x,A) \times C_1^{-}(y,B)}\nu_y(\mathcal{O}^{-}_r(\gamma^{-1}x,y))\nu_x(\mathcal{O}_r(\xi_0,x)).
         \end{aligned}
 \end{equation}
Let $$\mathcal{X'} = (C_{1}^{-}(x,A) \times C_{1}^{-}(y,B)) \cap (\widehat{U} \times \widehat{W})$$ and $$\mathcal{Y'} =(C_{1}^{-}(x,A) \times C_{1}^{-}(y,B)) \setminus (\widehat{U} \times \widehat{W}).$$ Then 

\begin{equation}
    \begin{aligned}
&\int^{T+3r}_{T_0}e^{\delta_{G} t}\sum_{\gamma \in G}\tilde{\mu}_G(K^{+} \cap g_{-t}\gamma K^{-})dt \\
&\geq \underbrace{r^2e^{-\frac{\epsilon}{12} }\sum_{\gamma \in G, T_0 + 3r \leq d(x,\gamma y) \leq T \atop (\gamma y, \gamma^{-1}x) \in \mathcal{X'}}\nu_y(\mathcal{O}^{-}_r(\gamma^{-1}x,y))\nu_x(\mathcal{O}_r(\xi_0,x))}_{III}\\
&+ \underbrace{r^2e^{-\frac{\epsilon}{12} }\sum_{\gamma \in G, T_0+3r \leq d(x,\gamma y) \leq T \atop (\gamma y, \gamma^{-1}x) \in \mathcal{Y'}}\nu_y(\mathcal{O}^{-}_r(\gamma^{-1}x,y))\nu_x(\mathcal{O}_r(\xi_0,x))}_{IV}.   
        \end{aligned}
\end{equation}
Remarks given above imply that 
\begin{equation}
    \begin{aligned}
        &III \geq \\
        &r^2e^{-\frac{7\epsilon}{60} }\sum_{\gamma \in G, T_0+3r \leq d(x,\gamma y) \leq T \atop (\gamma y, \gamma^{-1}x) \in \mathcal{X'}}\nu_y(\mathcal{O}_r(\eta_0,y))\nu_x(\mathcal{O}_r(\xi_0,x))\\
        &\geq e^{-\frac{7\epsilon}{60} }M \frac{e^{T\delta_G}}{\delta_G}\mu^{T}_{x,y}(\mathcal{X'})-C_1\\
        &\geq e^{-\frac{7\epsilon}{60} }M \frac{e^{T\delta_G}}{\delta_G}\mu^{T}_{x,y}(C^{-}_{1}(x,A) \times C^{-}_{1}(y,B)) -Me^{-\frac{7\epsilon}{60} }\|\mathcal{Y}_T'\| -C_1.
        \end{aligned}
\end{equation}

Where $\|\mathcal{Y}_T'\|$ is the cardinality of $$\mathcal{Y}_T'= \{\gamma \in G: d(s,\gamma y) \leq T, (\gamma y,\gamma^{-1}x) \in \mathcal{Y}'\}$$ and $C$ comes from the fact that, for fixed $x$ and $y$, there are only finitely many elements in $G$ with $d(x,\gamma y) \leq T_0+3r$. 

Now,  
\begin{equation}
    \begin{aligned}
        IV  \geq 0.
    \end{aligned}
\end{equation}
Notice that by Proposition \ref{lemma:subexponential}, $$\|\mathcal{Y}'_T\| = o(e^{T\delta_G}).$$  

Put those estimations together. Now the proof of Claim 4 can be completed exactly as the proof of Claim 3.

   \end{proof}
Thus we complete the proof of the proposition by Claim 3 and Claim 4.
\end{proof}

\begin{proposition}\label{prop:large}
   Let $\epsilon > 0$ and $x, y \in \T(S)$. Then for all $\xi_0, \eta_0 \in \PGM$, there exist $r > 0$ and open neighborhoods $U,W \subset \PGM$ containing $\xi_0$ and $\eta_0$,respectively, such that for all Borel subsets $A  \subset U, B \subset W$,
    \begin{equation*}
        \begin{aligned}
            \limsup_{R \to \infty} \mu^{R}_{x,y}(C^{-}_{r}(x,A) \times C^{-}_{r}(y,B)) \leq e^{\epsilon}\nu_x(A)\nu_y(B)/\lVert\mu_G(G)\lVert, \\
             \liminf_{R \to \infty} \mu^{R}_{x,y}(C^{+}_{r}(x,A) \times C^{+}_{r}(y,B)) \geq e^{-\epsilon}\nu_x(A)\nu_y(B)/\lVert\mu_G(G)\lVert.        \end{aligned}
    \end{equation*}

\end{proposition}

\begin{proof}
The proof also follows the same line as in the proof of \cite[Proposition 8.3]{Link} combined with strategy used in Proposition \ref{prop:small} and with \cite[Lemma 6.6(a)]{Link} be replaced by Proposition \ref{lemma:subexponential}(1).\marginpar{\textcolor{blue}{Note A.6}}
\end{proof}

\begin{proof}[Proof of Theorem \ref{thm:equidistribution}]
Let $x,y \in \T(S)$. We first construct a finite cover of $\GM$ as follows. For any $\xi \in \PGM$, take an open neighborhood $U_{\xi}$ of $\xi$ as in Lemma \ref{lemma:approx}. For any $x \in \T(S)$, take a ball of radius $1$. Then we get a open cover of $\GM$. Since $\GM$ is compact, there is a finite subcover 
$\mathfrak{C} = \lbrace \mathcal{P} \rbrace \cup \lbrace \widehat{W}_n \rbrace_{n \in \Gamma}$ of $\GM$, where $\mathcal{P}$ is the complement of the union of $\widehat{W}_n$ in $\GM$ which is relatively compact in $\T(S)$ and $\widehat{W}_n \cap \PGM \neq \emptyset$ for all $n$. Consider the associated finite cover $\mathfrak{C}' \times \mathfrak{C}'$ of $\GM \times \GM$.  Using a partition of unit subordinate to this finite cover, it is enough for us to prove the statement for any positive function compactly supported on one of these open subsets. Notice that for functions compactly supported on open subsets in one of $\mathcal{P} \times \mathcal{P}, \mathcal{P} \times \widehat{W}_n, \widehat{W}_n \times \mathcal{P}$, one can confirm the statement easily. So we only consider open subsets of the form $\widehat{W}_i \times \widehat{W}_j$.

Let $\widehat{U} \times \widehat{W}$ be such open subsets. Let $\widehat{U} \cap \PMF(S) = U$ and $\widehat{W} \cap \PGM = W$. Notice that, by our construction of the finite cover, the proofs of Proposition \ref{prop:small} and Proposition \ref{prop:large}, conclusions of these two propositions hold for $U$ and $W$. Hence a number $r$ exists so that the conclusion of Proposition \ref{prop:large} holds. Let $\mathcal{A}, \mathcal{B}$ be two Borel subsets of $\GM$ with closure $\overline{\mathcal{A}} \subset \widehat{U}, \overline{\mathcal{B}} \subset \widehat{W}$ such that $\partial (\mathcal{A} \times \mathcal{B})$ has null $\nu_x \otimes \nu_y-$measure. 

Let $\alpha > 0$. Choose open subsets $A_{+}, B_{+}$ and compact subsets $A_{-}, B_{-}$ of $\PGM$ such that 
\begin{equation}
    \begin{aligned}
        A_{-} \subset \mathcal{A}^{\circ} \cap \PGM \subset \overline{\mathcal{A}} \cap \PGM \subset A_{+} \subset U,\\
        B_{-} \subset \mathcal{B}^{\circ} \cap \PGM \subset \overline{\mathcal{B}} \cap \PGM \subset B_{+} \subset W,\\
   \nu_x(\mathcal{A}^{\circ} \cap \PGM - A_{-})< \alpha, \nu_x( A_{+}-\overline{\mathcal{A}} \cap \PGM) < \alpha,\\    \nu_x(\mathcal{B}^{\circ} \cap \PGM - B_{-}) < \alpha, \nu_x( B_{+}- \overline{\mathcal{B}} \cap \PGM) < \alpha.          \end{aligned}
\end{equation}
Notice that 
\begin{small}
\begin{equation}
\begin{aligned}
&\overline{\mathcal{A}} \times \overline{\mathcal{B}}=\\&((\overline{\mathcal{A}} \times \overline{\mathcal{B}}) - (C^{-}_{r}(x,A_{+}) \times C_{r}^{-}(y,B_{+})) \cup
 ((\overline{\mathcal{A}} \times \overline{\mathcal{B}}) \cap (C^{-}_{r}(x,A_{+}) \times C_{r}^{-}(y,B_{+})),\\
 &C^{+}_{r}(x,A_{-}) \times C_{r}^{+}(y,B_{-})= ((\mathcal{A} \times \mathcal{B}) \cap (C^{+}_{r}(x,A_{-}) \times C_{r}^{+}(y,B_{-})) 
\cup \\&((C^{+}_{r}(x,A_{-}) \times C_{r}^{+}(y,B_{-}) -(\mathcal{A} \times \mathcal{B})).
\end{aligned}
\end{equation}
\end{small}
By Proposition \ref{lemma:subexponential}, Proposition \ref{prop:large} and the fact that $\nu_x \otimes \nu_y(\partial (\mathcal{A} \times \mathcal{B})) = 0$, we have 
\begin{equation}
    \begin{aligned}
        \limsup_{R \to \infty}\mu^{R}_{x,y}( \mathcal{A} \times \mathcal{B}) &\leq \limsup_{R \to \infty}\mu^{R}_{x,y}((C^{-}_{r}(x,A_{+}) \times C_{r}^{-}(y,B_{+}))\\
        &\leq \frac{e^{\epsilon}\nu_x(A_{+})\nu_y(B_{+})}{\lVert\mu_G(G)\lVert}\\
        &\leq \frac{e^{\epsilon}(\nu_x(\overline{\mathcal{A}})\nu_y(\overline{\mathcal{B}})+\alpha(\lVert \nu_x\lVert +\lVert \nu_y\lVert))}{\lVert\mu_G(G)\lVert}\\
        &\leq \frac{e^{\epsilon}(\nu_x(\mathcal{A})\nu_y(\mathcal{B})+\alpha(\lVert \nu_x\lVert +\lVert \nu_y\lVert))}{\lVert\mu_G(G)\lVert}.
        \end{aligned}
\end{equation}
Take $\alpha$ small enough, one has 
$$\limsup_{R \to \infty}\mu^{R}_{x,y}( \mathcal{A} \times \mathcal{B}) \leq \frac{e^{\epsilon}\nu_x(\mathcal{A})\nu_y(\mathcal{B})}{\lVert\mu_G(G)\lVert}.$$
One can obtain similarly that  $$\limsup_{R \to \infty}\mu^{R}_{x,y}( \mathcal{A} \times \mathcal{B}) \geq \frac{e^{-\epsilon}\nu_x(\mathcal{A})\nu_y(\mathcal{B})}{\lVert\mu_G(G)\lVert}.$$
Hence for a continuous function $f$ which is compactly supported on $\widehat{U} \times \widehat{W}$, 
\begin{equation}
    \begin{aligned}
        &e^{-\epsilon}\int_{\PGM \times \PGM} f d(\nu_x\otimes \nu_y)\\
        &\leq \liminf_{R \to \infty}\int_{\PGM \times \PGM} fd(\mu^{R}_{x,y}) \\
        &\leq \limsup_{R \to \infty}\int_{\PGM \times \PGM} fd(\mu^{R}_{x,y}) \\
        &\leq  e^{\epsilon}\int_{\PGM \times \PGM} f d(\nu_x\otimes \nu_y).        
        \end{aligned}
\end{equation}
So the proof is completed.
\end{proof}
\section{Applications}
We present several cases where Theorem \ref{main:counting} can be applied.\\

\noindent\textbf{Mapping class groups.}~~If $G$ is the full mapping class group $\M(S)$, then one can use $(3)$ in Theorem \ref{main:counting}  to get an asymptotic growth formula for the mapping class group with respect to the Teicm\"{u}ller metric,  since in this case $\delta_G = 6g-6$ is the entropy of the Techm\"{u}ller geodesic flow $g_t$ on the unit tangent bundle $QD^1(S)$ of $\T(S)$ and $\mu_{G}$ is, up to a scale, the Masur-Veech measure constructed by Masur \cite{MasurInterval} and Veech \cite{Veech} which is finite. This is the main result of \cite{ABEM} which is inspired by the method in Margulis's thesis \cite{Margulis}. Our method here is however modelled on the ones given by Roblin \cite{Roblin} and Link \cite{Link}, so we obtain another proof for the case $\M(S)$.

\noindent\textbf{Convex cocompact subgroups.}~~A subgroup $G \leq \M(S)$ is called \textit{convex cocompact} if some $G$-orbit in $\T(S)$ is quasi-convex (see \cite{FarbMosher},\cite{KentLeininger}). Convex cocompact subgroups of $\M(S)$ are analogues of convex cocompact subgroups of Kleinian groups. It also has many interesting properties, e.g. $\Lambda G$ consists of uniquely ergodic measured foliations and it is virtually purely pseudo-Anosov. An important feature for a convex cocompact subgroup $G$ of $\M(S)$ is that there is a $G$-invariant subset $W_G$ of $\overline{\T(S)}^{Th}$ so that $G$ acts on $W_G$ cocompactly. This shows that the corresponding Bowen-Margulis measure $\mu_G$ is finite, so one can use $(3)$ in Theorem \ref{main:counting} to get an asymptotic formula for orbit growth. In fact, under the assumption that at least pseudo-Anosov element has an axis in the principal stratum, this formula is obtained in the first named author's thesis \cite{gekhtman2013dynamics} where a large part of the proof is quite different from the current paper. 

\noindent\textbf{Other examples.}~~

 Statistically convex cocompact (SCC) subgroups for general actions with contracting elements were introduced by Yang \cite{Yang_scc}.
 In the setting of mapping class group actions on Teichmueller space they are defined as follows. Let $G<\M(S)$, $o\in \T(S)$ and a constant $M>0$. Let $O_{M}$ denote the set of $g\in G$ such that some geodesic from $B(o, M)$ to $B(go,M)$ lies outside of the $M$-neighborhood of the orbit $Go$. We say $G$ is SCC if for some $M>0$ the set $O_{M}$ has strictly smaller exponential growth rate than $G$ itself (this definition can be shown to be independent of the basepoint $o\in \T(S)$). 
 In an upcoming work \cite{CGTY} it is shown that any SCC subgroup has finite Bowen-Margulis measure, and thus our results apply to SCC subgroups of mapping class groups. 

 Examples of SCC subgroups of mapping class groups which are not convex cocompact can be obtained as free products of a free abelian group $H$ generated by Dehn twists and a cyclic group generated by a single pseudo-Anosov $p$. See \cite[Proposition 6.6]{Yang_scc}.

\newpage
\appendix
\input{Appendix}

\newpage
\bibliography{ref}
\bigskip
\noindent
\noindent 
Ilya Gekhtman\\ \textbf{ilyagekh@gmail.com}\\

\medskip
\noindent
Biao Ma\\\textbf{biaoma@campus.technion.ac.il}

\end{document}

%% file: Appendix.tex
\section{Notes to aid the referee}
In this appendix, we add some notes to aid the referee and to make this paper much more self-contained.\\

\hrule

\medskip
\noindent{\bf Note A.1:} Denote $$A_n = \{r > 0: \nu_o(\tilde{\partial}\mathcal{O}_r(\xi,o)) > \frac{1}{n}\}.$$ Since $\nu_o$ is finite and different values for $r$ give disjoint $\tilde{\partial}\mathcal{O}_r(\xi,o)$, for any fixed natural number $n > 0$, $A_n$ is finite. Now observe that $$\{ r > 0: \nu_o(\tilde{\partial}\mathcal{O}_r(\xi,o)) > 0 \} = \bigcup_n A_n,$$ hence at most countable.
\medskip
\hrule

\noindent{\bf Note A.2:}
By our construction, $\nu_o(\PGM) = \nu_o(\mathcal{UE})= 1$. So according to Lemma \ref{lemma:measureclosure}, one has 
\begin{equation}
    \begin{aligned}
        &\nu_o(\mathcal{O}_r(\xi,o))\\
        &\leq \nu_o(\liminf_{n\to \infty}(\mathcal{O}^{\pm}_r(y_n,o)))\\
        &\leq \liminf_{n\to \infty}\nu_o(\mathcal{O}^{\pm}_r(y_n,o)))\\
        &\leq \limsup_{n\to \infty}\nu_o(\mathcal{O}^{\pm}_r(y_n,o)))\\
        &\leq \nu_o(\limsup_{n\to \infty}(\mathcal{O}^{\pm}_r(y_n,o)))\\
        &\leq \nu_o((\mathcal{O}_r(\xi,o) \cup \tilde{\partial}\mathcal{O}_r(\xi,o)) )\\
        &= \nu_o(\mathcal{O}_r(\xi,o)).       
        \end{aligned}
\end{equation}
Thus we have, as claimed, 
\begin{equation}
    \begin{aligned}
        &\lim_{n \to \infty}\nu_o(\mathcal{O}^{\pm}_{r}(y_n,o))\\
        &=\nu_o(\mathcal{O}_r(\xi,o)).
            \end{aligned}
\end{equation}
\hrule
\medskip
\noindent{\bf Note A.3:}
The reader can check part (1) easily by definitions. We now consider part (2). 

Take $\zeta \in \mathcal{O}_r(\eta,x)$, then by definition, the geodesic $\mathcal{L}(\eta,\zeta)$ has distance $r$ with $x$. Reverse the direction of $\mathcal{L}(\eta,\zeta)$, one gets $\mathcal{L}(\zeta,\eta)$. Then one can check directly that $(\zeta,\eta) \in L_{r}(x,\gamma y)$. In order to show that $\xi \in \gamma B$ and $\eta \in A$, by the definition of $C^{-}_r(x,A)$, $\gamma y \in C^{-}_{r}(x,A)$ which gives that $\eta \in \mathcal{O}^{-}_r(x,\gamma y) \subset A$. Since $(\zeta,\eta) \in L_r(x,\gamma y)$, we have $\zeta \in \mathcal{O}_{r}^{+}(\gamma y,x)$ which implies that $\gamma^{-1}\zeta \in \mathcal{O}^{+}_{r}(y,\gamma^{-1}x)$. Here we use the fact that $G$ acts on $\T(S) \cup_{GM} \mathcal{UE}$ continuously. Since $\gamma^{-1}x \in C^{-}_{r}(y,B)$, by the definition of $C^{-}_{r}(y,B)$, we have $\mathcal{O}^{+}_{r}(y,\gamma^{-1} x)$. This implies that $\xi \in \gamma B$.

\newpage

\noindent{\bf Note A.4:}

\underline{Claim 1:} $ \Psi'(\bigcup_{t > 0}\{K_A^{+} \cap g_{-t}\gamma K_B^{-}\}) \subset L_r(x, \gamma y) \cap (\gamma B \times A)$.\\
Let $q \in K_A^{+} \cap g_{-t}\gamma K_B^{-}$ for some $t > 0$, then by definition of $K^{\pm}$ (Eqs \ref{equation:K}), $\Psi'(q) := (\xi,\eta) \in \mathcal{UE}\times \mathcal{UE}$ with  $\xi \in \gamma B$ and $\eta \in A$. Denote $q = w(x;\xi,\eta)$ and let $\tau \in \mathds{R}$ such that $$w(\gamma y; \xi,\eta) = g_{t}q.$$ Now on the one hand, by definition of $K^{\pm}$ (Eqs \ref{equation:K}), $\|d(x,\gamma y) - \tau\| \leq 2r$ which together with the assumption that $d(x,\gamma y) \geq 12r$ shows that $\tau  > 0$. On the other hand, $d(x,\gamma y) \geq 12r$ also implies that $md(\Pi_{\mathcal{L}(\xi,\eta)}(x),\Pi_{\mathcal{L}(\xi,\eta)}(\gamma y)) > r$. So we can conclude that $\Psi'(q) \subset L_{r}(x,\gamma y)$.\\
\underline{Claim 2:} $ \Psi'(\bigcup_{t > 0}\{K_A^{+} \cap g_{-t}\gamma K_B^{-}\}) \supset L_r(x, \gamma y) \cap (\gamma B \times A)$.\\
Let $$(\xi,\eta) \in L_r(x, \gamma y) \cap (\gamma B \times A).$$ Consider the geodesic $\ell = \mathcal{L}(\xi,\eta)$. Then by definition, there are two points $m,n \in \ell$ with $g_t m = n$ for some $t > 0$, such that $$\Psi'(\ell) = (\xi,\eta), d(x,m) < r, d(ry,n) < r.$$ Set $q = w(x;\xi,\eta)$ and $\tau \in \mathds{R}$ such that $g_{\tau} q = w(\gamma y;\xi,\eta)$. One then can check easily that $$q \in K_A^{+}, g^{\tau} q \in \gamma K_B^{-}, \tau > 0.$$

\hrule
\medskip
\noindent{\bf Note A.5:}
We supply computations here.

By the above mixing, there exists $T_0 > 12r$ such that for $t \geq T_0$, we have 
\begin{equation}\label{equation:consequencemixing}
    \begin{aligned}
        &e^{-\frac{\epsilon}{3}}\tilde{\mu}_G(K^{+})\tilde{\mu}_G(K^{-}) \leq \|\mu_G\|\sum_{\gamma \in G}\tilde{\mu}_G(K^{+} \cap g_{-t}\gamma K^{-})\\
        &\leq e^{\frac{\epsilon}{3}}\tilde{\mu}_G(K^{+})\tilde{\mu}_{G}(K^{-}).
        \end{aligned}
\end{equation}
By (\ref{equ:measureofK}) and (\ref{equ:measureofO}), we have, for $A \subset \widehat{V}$ and $B \subset \widehat{W}$,
\begin{equation}
    \begin{aligned}
        &re^{-\frac{\epsilon}{30}}\nu_x(\mathcal{O}_r(\xi_0,x))\nu_x(A) \leq \tilde{\mu}_G(K^{+}) \leq re^{2r\delta_G}e^{\frac{\epsilon}{30}}\nu_x(\mathcal{O}_r(\xi_0,x)\nu_x(A),\\
        &re^{-\frac{\epsilon}{30}}\nu_y(\mathcal{O}_r(\eta_0,y))\nu_y(B) \leq \tilde{\mu}_G(K^{-}) \leq re^{2r\delta_G}e^{\frac{\epsilon}{30}}\nu_y(\mathcal{O}_r(\eta_0,y)\nu_y(B).    \end{aligned}
\end{equation}
Denote $M = r^2\nu_x(\mathcal{O})_r(\xi_0,x))\nu_y(\mathcal{O}_r(\eta_0,y)) > 0$ and since $r\delta_G \leq \frac{\epsilon}{60}$, 
\begin{equation*}
    \begin{aligned}
        &e^{-\frac{\epsilon}{15}}M\nu_x(A)\nu_y(B) \leq \tilde{\mu}_G(K^{+})\tilde{\mu}_G(K^{-})\\
        &\leq e^{\frac{\epsilon}{5}}M\nu_x(A)\nu_y(B).
    \end{aligned}
\end{equation*}
By Inequalities (\ref{equation:consequencemixing}), we have for $t \geq T_0$, 
\begin{equation*}
    \begin{aligned}
        &M\nu_x(A)\nu_y(B) \leq e^{\frac{2\epsilon}{5}} \|\mu_G\|\sum_{\gamma \in G}\tilde{\mu}_G(K^{+} \cap g_{-t}\gamma K^{-}),\\
        &M\nu_x(A)\nu_y(B) \geq e^{-\frac{8\epsilon}{15}} \|\mu_G\|\sum_{\gamma \in G}\tilde{\mu}_G(K^{+} \cap g_{-t}\gamma K^{-}).        
        \end{aligned}
\end{equation*}
Thus 
\begin{equation}
    \begin{aligned}
&(e^{\delta_G(T-6r)}-e^{\delta_G T_0})M\nu_x(A)\nu_x(B) = \delta_G \int^{T-6r}_{T_0}e^{t\delta_G}M\nu_x(A)\nu_x(B)dt \\
&\leq e^{\frac{2\epsilon}{5}}\lVert \mu_G \lVert \delta_G \int^{T-6r}_{T_0}e^{\delta_G t}\sum_{\gamma \in G}\tilde{\mu}_G(K^{+} \cap g_{-t}\gamma K^{-}),\\
    &(e^{\delta_G(T+3r)}-e^{\delta_G T_0})M\nu_x(A)\nu_x(B) = 
    \delta_G \int^{T+3r}_{T_0}e^{t\delta_G}M\nu_x(A)\nu_x(B)dt\\
    &\geq e^{-\frac{8\epsilon}{15}}\lVert \mu_G \lVert \delta_G \int^{T+3r}_{T_0}e^{\delta_G t}\sum_{\gamma \in G}\tilde{\mu}_G(K^{+} \cap g_{-t}\gamma K^{-}). 
    \end{aligned}
\end{equation}

\hrule
\medskip
\noindent{\bf Note A.6:}
Let $(\xi_0,\eta_0) \in (\PGM)^2$. Choose conical points $\xi^{+}$ and $\eta^{+}$ in the support of the measures. Take $x_0 \in \mathcal{L}(\xi_0,\xi^{+})$ and $y_0 \in \mathcal{L}(\eta_0,\eta^{+})$ so that they have trivial stabilizers in $G$. Apply Proposition \ref{prop:small} to $(x_0,y_0,\xi_0,\eta_0)$ to get desired open neighborhoods $\widehat{U}_0,\widehat{W}_0$ of $\xi_0$ and $\eta_0$, where $\epsilon$ is taken to be $\frac{\epsilon}{3}$.

Choose an open neighborhood $U_0$ of $\xi_0$ in $\GM$ so that $U_0 \cap \PGM \subset \widehat{U}_0$. Moreover $U_0$ is chosen so that if $a \in U_0 \cap \mathcal{UE}$, then $|\beta_a(x_0,x)-\beta_{\xi_0}(x_0,x)| \leq \frac{\epsilon}{6\delta_G}$, if $a \in \T(S)$,
\begin{equation}\label{equ:cocycle}
    \begin{aligned}
       |d(x_0,a)-d(x,a)- \beta_{\xi_0}(x_0,x)| \leq \frac{\epsilon}{6\delta_G}.    \end{aligned}
\end{equation} Choose an open neighborhood $W_0$ of $\eta_0$ in the same way. 

Let $U \subset \PGM$,$W \subset \PGM$ be open neighborhoods of $\xi_0$ and $\eta_0$, respectively, so that $$\overline{U} \subset U_0 \cap \PGM, \overline{W} \subset W_o \cap \PGM.$$ Take $A \subset U, B \subset W$ be arbitrary Borel sets.

On the one hand, notice that if $\xi \in A \cap \mathcal{UE}$, then $\beta_{\xi_0}(x_0,x) < \beta_{\xi}(x_0,x) + \frac{\epsilon}{6\delta_G}$. Hence by conformality of the measure,
\begin{equation}\label{equ:appeconformal}
    \begin{aligned}
        &e^{\delta_G \beta_{\xi_0}(x_0,x)}\nu_{x_{0}}(A)\\
        &= \int_Ae^{\delta_G}\beta_{\xi_0}(x_0,x)d\nu_{x_0}(\xi)
        &\leq e^{\frac{\epsilon}{6}}\int_{A}d\nu_x(\xi) = e^{\frac{\epsilon}{6}}\nu_x(A).    \end{aligned}
\end{equation}
One can also deduce similiarly that $$e^{\delta_G \beta_{\xi_0}(y_0,y)}\nu_{y_{0}}(B) \leq e^{\frac{\epsilon}{6}}\nu_y(B).$$

On the other hand, set $r = 1 + \max\{d(x,x_0),d(y,y_0)\}$ and $$V_{-r} = \{x \in \T(S): \overline{B_r(z)} \subset U_0\} \cup (U_0 \cap \PGM).$$
If $d(x,\gamma y) \leq T$ and $(\gamma y, \gamma^{-1}x) \in V_{-r} \times W_{0}$, then $$(\gamma y_0, \gamma^{-1}x) \in U_0 \times W_0.$$ By the definition of $r$, and one then can check, by using (\ref{equ:cocycle}) and the one for $y_0$, that 
\begin{equation}
    \begin{aligned}
        &d(x_0,\gamma y_0) \leq d(y,\gamma^{-1}x) + \beta_{\eta_0}(y_0,y) + \beta_{\xi_0}(x_0,x) + \frac{\epsilon}{3\delta_G}\\
        &\leq T + \beta_{\eta_0}(y_0,y) + \beta_{\xi_0}(x_0,x) + \frac{\epsilon}{3\delta_G}.    \end{aligned}
\end{equation}
Let $$\mathcal{X}_1 = C^{-}_{r}(x,A) \times C^{-}_{r}(y,B) \cap V_{-r} \times W_{0}$$, $$\mathcal{X}_2 = C^{-}_{r}(x,A) \times C^{-}_{r}(y,B) \setminus V_{-r} \times W_{0},$$ and $$\mathcal{X}_3 = C^{-}_{1}(x,A) \times C^{-}_{1}(y,B) \cap V_{0} \times W_0$$
We then have, for $T\gg 1$,
\begin{equation}
    \begin{aligned}
   & \{\gamma \in G: d(x,\gamma y) \leq T, (\gamma y, \gamma^{-1}x) \in \mathcal{X}_1\} \\
   &\subset \{\gamma \in G: d(x_0,\gamma y_0) \leq T+ \beta_{\eta_0}(y_0,y) + \beta_{\xi_0}(x_0,x) + \frac{\epsilon}{3\delta_G}, (\gamma y_0, \gamma^{-1}x_0) \in \mathcal{X}_3\}    \end{aligned}
\end{equation}
 Notice that by Proposition \ref{lemma:subexponential}, the number of elements in $\mathcal{X}_2$ is $o(e^{T\delta})$. Therefore, by Proposition \ref{prop:small},
 \begin{equation}
     \begin{aligned}
         &\limsup_{T \to \infty}\mu^{T}_{x,y}(C_r^{-}(x,A) \times C^{-}_{r}(y,B))\\
         &\leq e^{\frac{\epsilon}{3}}e^{\delta_G(\beta_{\xi_0}(x_0,x) + \beta_{\eta_0}(y_0,y))}(\\
         &\limsup_{T \to \infty}\mu^{T + \beta_{\eta_0}(y_0,y) + \beta_{\xi_0}(x_0,x) + \frac{\epsilon}{3\delta_G} }_{x_0,y_0}(C_1^{-}(x_0,A) \times C^{-}_{1}(y_0,B)) )\\
         &\leq \frac{1}{\|\mu_G\|} e^{\frac{2\epsilon}{3}}e^{\delta_G(\beta_{\eta_0}(y_0,y) + \beta_{\xi_0}(x_0,x))}\nu_{x_0}(A)\nu_{y_0}(B).
         \end{aligned}
 \end{equation}
Combine with (\ref{equ:appeconformal}) and the one for $\nu_y(B)$, we can complete the proof for the limit superior. For the limit inferior, the argument is similar.